\documentclass[12pt]{article}
	
	\usepackage[margin=1in]{geometry}  
	\usepackage{graphicx}              
	\usepackage{amsmath}               
	\usepackage{amsfonts}              
	\usepackage{amsthm}                
	\usepackage{mathtools}
	\usepackage{xcolor}
	\usepackage{hyperref}
	\usepackage{tikz-cd}
    \usepackage{amssymb}
	
	\newtheorem{thm}{Theorem}[section]
	\newtheorem{lem}[thm]{Lemma}
	\newtheorem{prop}[thm]{Proposition}
	\newtheorem{cor}[thm]{Corollary}
	
	\newtheorem{defn}[thm]{Definition}
	
	\newtheorem{remark}[thm]{Remark}

	\newcommand{\RR}{\mathbb{R}}      

    \newcommand{\Ga}{\mathrm{G}}
    \newcommand{\Pol}{\operatorname{Pol}}

\begin{document}

\title{Geodesics for generalised Monge-Amp\`{e}re equations}

\author{Gao Chen, Kartick Ghosh, and Sijie Nie}
\maketitle

\begin{abstract}
  Xiuxiong Chen proved the existence of $C^{1, \bar 1}$ geodesics in the space of K\"{a}hler potentials and the convexity of the 
  $J$-functional along $C^{1, \bar 1}$ geodesics. Analogous results were proved by Collins and Yau for the hypercritical 
  Leung--Yau--Zaslow equation. 
  In this paper, we study a similar picture for generalised Monge--Amp\`{e}re equations associated with two right-Noetherian polynomials in proper position.
\end{abstract}

\tableofcontents 

\section{Introduction}

The story of geodesics in the space of K\"ahler potentials starts with Xiuxiong Chen's breakthrough paper \cite{ChenKahlerMetrics}, which proves the existence of $C^{1,\bar 1}$ geodesics connecting two K\"ahler potentials. The later work of Chu, Tosatti, and Weinkove \cite{ChuTosattiWeinkoveGeodesics} proves the full $C^{1,1}$ estimate. The counterexample of Darvas and Lempert \cite{DarvasLempertWeakGeodesics} proves that the $C^{1,1}$ estimate cannot be improved to $C^2$. In another direction, Xiuxiong Chen studied the J-equation proposed by Donaldson \cite{DonaldsonMomentMaps} from the viewpoint of K\"ahler geometry \cite{ChenMabuchiEnergy}. Its solutions are critical points of the J-functional, which is convex along geodesics.

From the point of view of mirror symmetry, the J-equation is also the large-radius limit of the deformed Hermitian--Yang--Mills equation proposed by Leung, Yau, and Zaslow \cite{LeungYauZaslowFourierMukai}. We call this the LYZ equation in this paper. In analogy with Xiuxiong Chen's theory, Collins and Yau \cite{CollinsYauGeodesics} study geodesics on a connected component of the space $\{\varphi: \operatorname{Re}(e^{-\sqrt{-1}\theta}(\omega+\sqrt{-1}\chi_\varphi)^n)>0\}$ and prove the $C^{1,\bar 1}$ estimate. Collins and Yau \cite{CollinsYauGeodesics} also study the functional for the hypercritical LYZ equation $\operatorname{Im}(e^{-\sqrt{-1}\theta}(\omega+\sqrt{-1}\chi_\varphi)^n)=0$, and they prove that this functional is convex along their geodesics. The estimate for the geodesic was improved to $C^{1,1}$ by Chu, Collins, and Lee \cite{ChuCollinsLeeAlmostCalibrated}. 

The J-equation is also used in the continuity path proposed by Xiuxiong Chen \cite{ChenCscKContinuityPath} to study constant scalar curvature K\"ahler metrics. The breakthrough works of Xiuxiong Chen and Jingrui Cheng \cite{ChenChengCscK1, ChenChengCscK2} and Berman-Darvas-Lu \cite{BermanDarvasLu} prove that, after normalizing potentials and accounting for holomorphic automorphisms, the existence of cscK metrics is equivalent to the coercivity of Mabuchi's K-energy functional \cite{MabuchiKEnergy} with respect to the $d_1$ distance on the $\mathcal{E}^1$ space introduced in \cite{GuedjZeriahiFiniteEnergy, DarvasFiniteEnergyGeometry, DarvasMabuchiCompletion}. For the J-equation, a properness result is proved in \cite{CollinsGabor}. For the hypercritical LYZ equation, the theory of $\mathcal{E}^1$ spaces is established in \cite{ChuCollinsLeeAlmostCalibrated}, \cite{ChuLeeHypercriticalDHYM}, and the analogue of the Chen--Cheng result is proved in \cite{ChuLeeHypercriticalDHYM}.

In this paper, we study more general equations. Using Lin's terminology \cite{LinInverseSigmaConvexity}, we call the associated polynomials right-Noetherian. In particular, our theory includes the supercritical, but not necessarily hypercritical, phase of the LYZ equation. More recently, Fang and Ma \cite{FangMaGarding} extended Lin's theory to a larger class of polynomials. Using Lin's estimates \cite{LinInverseSigmaSolvability} and Fang and Ma's theory \cite{FangMaGarding}, we study two right-Noetherian polynomials $f_1$ and $f_2$ such that $f_2\prec f_1$ in the sense of \cite{FangMaGarding}.

Our first main theorem is the following:
\begin{thm}
\label{MainThm1}
Let $(X,\omega)$ be a compact K\"ahler manifold and let $\chi$ be a real closed smooth $(1,1)$-form. Set $c_n^{(2)}=1$, $c_{n-1}^{(2)}=0$ and let \[f_2=\sum_{k=0}^{n}c_k^{(2)}\binom{n}{k}x^k, \quad g_2(\lambda)=\sum_{k=0}^{n}c_k^{(2)}\sigma_k(\lambda)\] be a right-Noetherian polynomial and its polarization, and let
\[
\mathcal{H}=\{\varphi\in C^{\infty}(X,\mathbb{R}): \chi_\varphi=\chi+\sqrt{-1}\partial\bar\partial\varphi\in \Upsilon_{g_2}(\omega)\}
\]
be the space of admissible potentials. Set $\mathcal{A}=\{z_0\in\mathbb{C}:1<\lvert z_0\rvert<e\}$ and $t=\log\lvert z_0\rvert$. Then, for any $\varphi_0, \varphi_1\in \mathcal{H}$ and any sufficiently small $\epsilon>0$, there exists a smooth $S^1$-invariant $\epsilon$-geodesic $\varphi_{\epsilon}$ on $\overline{\mathcal{X}}=X\times\overline{\mathcal{A}}$. Moreover, $\varphi_{\epsilon}$ is uniformly bounded and $\varphi_{\epsilon_1}\ge\varphi_{\epsilon_2}$ for all $0<\epsilon_1\le\epsilon_2$.
\end{thm}

We have discussed the $C^{1,\bar 1}$ estimate with ChatGPT 5.6 Sol \[
||\sqrt{-1}D\bar{D}\varphi_\epsilon||_{L^{\infty}(X\times \mathcal{A})}\le C,
\] where the norm is measured in any fixed K\"ahler metric on $\mathcal{X}$. AI claims that this can be proved using the spatial eigenvalue method of \cite{CollinsYauGeodesics} if an algebraic lemma holds. It has tested that lemma numerically and cannot find any counterexample, but it only claims the proof of the lemma for the supercritical LYZ equation. Since its claim is long and hard to read, we will not present anything here, and we still view the $C^{1,\bar 1}$ estimate as an open problem.

However, even without the $C^{1,\bar 1}$ estimate, we can define the weak geodesic $\varphi_t$ in the sense of Bedford-Taylor \cite{BedfordTaylorCapacity}, and define the mixed Aubin-Mabuchi energy. Our second main theorem proves that the $J$-functional for $f_1$ is convex along geodesics for $f_2$, and the $I$-functional for $f_2$ is affine along geodesics for $f_2$.
\begin{thm}
\label{MainThm3}
Let $f_1$ and $f_2$ be two right-Noetherian polynomials satisfying $f_2\prec f_1$. If $I_2$ is the functional associated with $f_2$, and $J_1$ is the functional associated with $-f_1$, then $J_1(t)$ is Lipschitz and convex, and $I_2(t)$ is affine along the geodesic $\varphi_t$ connecting any two potentials in $\mathcal{H}$.
\end{thm}

We can also prove that the $d_p$ length of the $\epsilon$-geodesics converges to the $d_p$ distance. As a byproduct, we obtain an estimate for the $d_p$ distance, which implies that $(\mathcal{H}, d_p)$ is a metric space. We call the completion the $\mathcal{E}^p$ space, which is the classical one when $f_2(x)=x^n$. See \cite{DarvasMabuchiCompletion} for more details.

\begin{thm}
\label{MainThm2}
For any $p\in[1,\infty)$ and any $\varphi_{0},\varphi_{1}\in \mathcal{H}$, let $\varphi_{\epsilon}$ be the $\epsilon$-geodesic connecting $\varphi_{0}$ and $\varphi_{1}$. Then we have
\begin{enumerate}
    \item $d_{p}(\varphi_{0},\varphi_{1})=\lim_{\epsilon\rightarrow 0}\operatorname{length}_{p}(\varphi_{\epsilon})$;
    \item The following inequality holds 
\[
d_p^p(\varphi_0,\varphi_1) \ge \max\Big\{
\int_{\{\varphi_0\ge\varphi_1\}} (\varphi_0-\varphi_1)^p g_2(\chi_{\varphi_0}),\;
\int_{\{\varphi_1\ge\varphi_0\}} (\varphi_1-\varphi_0)^p g_2(\chi_{\varphi_1})
\Big\}.
\]
\end{enumerate}
\end{thm}

Finally, let us compare our results with the recent work of the first and third authors with Xu \cite{ChenNieXuNumericalCriterion}, which proves a Nakai-Moishezon-type criterion following the roadmap of Demailly and P\u{a}un \cite{DemaillyPaunKahlerCone}, the first author \cite{ChenJEquationDHYM}, and others \cite{DatarPingaliNumericalCriterion, SongNakaiMoishezon, ChuLeeTakahashiNakaiMoishezon, FangMaDifferentialForms, FuZhangZhangTwoHessian}. In that paper, the projectivity assumption is essential, as shown by Zhang's counterexample \cite{ZhangSupercriticalDHYM}. In the K\"ahler case, the best result one can hope for is a Chen-Cheng-type coercivity theorem. This will be studied by the first two authors and another collaborator in \cite{ChenGhoshWangCoerciveness}.

The paper is organized as follows. Section 2 reviews the right-Noetherian and Fang--Ma--G\aa rding polynomial framework and derives several useful lemmas. Section 3 derives the geodesic equation, rewrites it on $X\times\mathcal A$, and introduces its $\epsilon$-regularization. Section 4 introduces a continuity path for the $\epsilon$-geodesic equation and constructs uniform sub- and supersolutions, yielding the $C^0$ and time-derivative estimates. Section 5 establishes fixed-$\epsilon$ boundary and interior complex-Hessian and gradient estimates. Section 6 uses the fixed-$\epsilon$ theory to obtain smooth $\epsilon$-geodesics and then takes their monotone limit to define the bounded weak geodesic. Section 7 defines the mixed Aubin--Mabuchi functionals and proves linearity of $I_2$ and convexity of $J_1$ along the weak geodesic. Section 8 introduces the 
$d_p$ metric, proves convergence of $\epsilon$-geodesic lengths to it, and defines its abstract metric completion $\mathcal{E}^p$. Section 9 study several examples.

\paragraph{\textbf{Declaration on the use of AI.}}
We have discussed the $C^{1,\bar 1}$ estimate with ChatGPT 5.6 Sol but cannot prove anything interesting to us. ChatGPT 5.6 Sol was also used to assist with presentations and grammar corrections of this paper.

\section{Preliminaries}

 This section recalls the required materials from
\cite{FangMaGarding} and \cite{LinInverseSigmaConvexity}, \cite{LinInverseSigmaConvexityCorrigendum} in notation adapted to the present paper and derives some consequences which will be used later.

A polynomial is \emph{multi-affine} if its degree in each individual
variable is at most one.

\begin{defn}[Definition 3.1(7),(9) of \cite{FangMaGarding}]
Let \[\sigma_k(x_1,\ldots,x_n)=\sum_{i_1<\cdots<i_k}x_{i_1}\cdots x_{i_k}\] be the \(k\)-th elementary symmetric
polynomial, with \(\sigma_0=1\).

If \(h(x)=\sum_{k=0}^{d}c_k \binom nk x^k\) and \(n\ge d\), its ordinary polarization is
\[
  \Pol_n(h)(x_1,\ldots,x_n)
  =
  \sum_{k=0}^{d}c_k\sigma_k(x_1,\ldots,x_n).
\]
It is the unique symmetric multi-affine polynomial satisfying
\[
  \Pol_n(h)(x,\ldots,x)=h(x).
\]

More generally, let
\(\boldsymbol{\kappa}=(\kappa_1,\ldots,\kappa_m)\in\mathbb N^m\), and
suppose \(h(x_1,\ldots,x_m)=\sum_{\alpha\le\boldsymbol{\kappa}}
c_\alpha x^\alpha\), where
\(\alpha=(\alpha_1,\ldots,\alpha_m)\),
\(x^\alpha=\prod_i x_i^{\alpha_i}\), and
\(\alpha\le\boldsymbol{\kappa}\) means
\(\alpha_i\le\kappa_i\) for every \(i\).  The
\(\boldsymbol{\kappa}\)-polarization is
\[
  \Pi^\uparrow_{\boldsymbol{\kappa}}(h)
  =
  \sum_{\alpha\le\boldsymbol{\kappa}}c_\alpha
  \prod_{i=1}^m
  \frac{\sigma_{\alpha_i}(x_{i1},\ldots,x_{i\kappa_i})}
       {\binom{\kappa_i}{\alpha_i}}.
\]
Its output is multi-affine in all the variables
\(\{x_{ij}:1\le i\le m,\ 1\le j\le\kappa_i\}\), and setting each
block \(x_{i1}=\cdots=x_{i\kappa_i}=x_i\) recovers \(h\).
\end{defn}

\begin{defn}[{Definitions 2.2 and 2.3} of \cite{FangMaGarding}]
For \(n\ge1\), write
\begin{align*}
  \Gamma_n^+
  &=\{x=(x_1,\ldots,x_n)\in\RR^n:x_i>0\text{ for every }i\},\\
  \overline{\Gamma_n^+}
  &=\{x=(x_1,\ldots,x_n)\in\RR^n:x_i\ge0\text{ for every }i\}.
\end{align*}
A set \(\Upsilon\subseteq\RR^n\) passes the \emph{positive ray test}
\emph{(PRT)} if
\[
  \Upsilon+\overline{\Gamma_n^+}\subseteq \Upsilon.
\]
It is \emph{negative ray terminating} \emph{(NRT)} if, for
every \(x\in \Upsilon\), the set
\[
  (\{x\}-\overline{\Gamma_n^+})\cap \Upsilon
\]
is bounded.
\end{defn}

We now introduce the definition of a G{\aa}rding polynomial from \cite{FangMaGarding}. To distinguish this notion from the classical one, we call such a polynomial a Fang-Ma-G{\aa}rding polynomial. Since the original definition in Definitions~4.8 and 4.9 of \cite{FangMaGarding} is more complicated, we use another characterization here, which is equivalent to the original definition by Theorem~1.1(3) of \cite{FangMaGarding}.

\begin{defn}[Theorem 1.1(3) of
\cite{FangMaGarding}]
\label{FMGdefinition}
Let \(p\in\RR[x_1,\ldots,x_n]\) be nonzero.  For a
multi-index
\(\alpha=(\alpha_1,\ldots,\alpha_n)\in\mathbb Z_{\ge0}^n\), write
\[
  \partial^\alpha p
  =
  \frac{\partial^{|\alpha|}p}
       {\partial x_1^{\alpha_1}\cdots\partial x_n^{\alpha_n}},
  \qquad |\alpha|=\alpha_1+\cdots+\alpha_n.
\]
Then \(p\) is called a \emph{Fang-Ma-G{\aa}rding polynomial} if and only if,
for every \(\alpha\) such that \(\partial^\alpha p\not\equiv0\),
\begin{enumerate}
  \item the set \(\{\partial^\alpha p>0\}\) has a unique connected
  component \(\Upsilon_{\partial^\alpha p}\) passing PRT; and
  \item this component is contained in the corresponding components
  of all its first partial derivatives:
  \[
    \Upsilon_{\partial^\alpha p}
    \subseteq \Upsilon_{\partial_i\partial^\alpha p}
    \qquad(1\le i\le n).
  \]
\end{enumerate}
Here \(\Upsilon_h=\RR^n\) when \(h\equiv c>0\), and by convention \(\Upsilon_0=\RR^n\).
The component \(\Upsilon_p\) is called the \emph{Fang-Ma-G{\aa}rding component} of
\(p\), which is a generalization of the $\Upsilon$-stable component of \cite{LinInverseSigmaConvexity}.  We denote the class of such polynomials in \(n\) variables by
\(\Ga_n\), and put \(\Ga=\bigcup_{n\ge1}\Ga_n\).

\end{defn}

\begin{thm}[{Theorem 8.13} of \cite{FangMaGarding}]
\label{kappa-polarization}
Let \(p\in\Ga_n\) have multidegree at most
\(\boldsymbol{\kappa}\), meaning
\(\deg_{x_i}p\le\kappa_i\) for every \(i\).  If its Fang-Ma-G{\aa}rding
component \(\Upsilon_p\) is NRT, then
\[
  \Pi^\uparrow_{\boldsymbol{\kappa}}(p)\in\Ga.
\]
\end{thm}

Fang and Ma have provided the following characterization of NRT polynomials:

\begin{lem}[Lemma 9.5 of \cite{FangMaGarding}]
\label{NRT-redundant}
Let $f\in \Ga_n$. Then the Fang--Ma--G{\aa}rding component $\Upsilon_f$ is NRT if and only if $\deg_{x_i} f\ge 1$ for every $i=1,2,\ldots,n$.
\end{lem}

We next recall the definition introduced by Lin.

\begin{defn}[Definition 2.1 of \cite{LinInverseSigmaConvexity}]
\label{right-Noetherian}
For a nonconstant univariate polynomial \(h\), let \(r(h)\) denote its
largest real root, provided such a root exists.  If
\(\deg h=d\), its \emph{root sequence} is
\[
  R(h)=\bigl(r(h),r(h'),\ldots,r(h^{(d-1)})\bigr),
\]
provided all entries exist.  The polynomial \(h\) is
\emph{right-Noetherian} if all these entries exist
and
\[
  r(h)\ge r(h')\ge\cdots\ge r(h^{(d-1)}).
\]

The polynomial \(h\) with degree $d\ge 2$ is
\emph{strictly right-Noetherian} if all these entries exist
and
\[
  r(h)> r(h')\ge\cdots\ge r(h^{(d-1)}).
\]

\end{defn}

We need the following lemma.
\begin{lem}
\label{IntegralUpsilonStable}
If $f(x) = \sum_{k=0}^{n}c_k \binom{n}{k} x^k$, with $c_n=1$, is right-Noetherian, then, for all sufficiently large $C$, 
\[\tilde f(x)=\int_0^x f(t)\,dt - C= \frac{1}{n+1} \Big( \sum_{k=0}^{n} c_k\binom{n+1}{k+1}x^{k+1} - (n+1)C \Big) \]
is strictly right-Noetherian.
\end{lem}
\begin{proof}
Using the notation in Definition~\ref{right-Noetherian}, we have
\[
r(f) \ge r(f') \ge \cdots\ge r(f^{(n-1)}).
\]
Since $\tilde f'(x) = f(x)$, we only need to show that for all sufficiently large $C$, $r(\tilde f)$ exists and
\[
r(\tilde f) > r(\tilde f') = r(f).
\]
Choose $C_0 := \int_0^{r(f)} f(t)\,dt$. Then, for all $C > C_0$, $\tilde f(r(f)) < 0$.
Moreover, $\lim_{x\to +\infty}\tilde f(x) = +\infty$, so the intermediate value theorem 
shows that there exists a point $x_0 > r(f)$ such that $\tilde f(x_0)=0$.
This implies that $r(\tilde f) > r(f)$ for all $C > C_0$.
\end{proof}

The following is immediate from Section~9.2 of \cite{FangMaGarding}.
\begin{lem}[Section 9.2 of \cite{FangMaGarding}]
\label{FangMaLinCorrespondence}
A monic univariate polynomial $p(x)$ belongs to $\Ga_1$ if and only if it is right-Noetherian.
\end{lem}

Next, let us recall the definition of proper position from \cite{FangMaGarding}.

\begin{defn}[{Definition 10.9} of \cite{FangMaGarding}]
 Suppose $f_2(x), f_1(x) \in \Ga[x] \setminus \{0\}$. The pair $(f_2,f_1)$ is said to be \emph{in proper position} if
 \[f_2(x)y +f_1(x) \in \Ga[x,y].\]
In this case, we also write $f_2 \prec f_1$.
\end{defn}

Using this definition and Theorem \ref{kappa-polarization}, it is almost trivial to prove the following corollary. However, since it is not explicitly stated in their paper, we write it out here for completeness:

        \begin{cor}
        \label{polarizationproperposition}
        Suppose that $f_2, f_1\in \Ga_1$ are univariate monic polynomials of degree $1\le d\le n$. If $f_2\prec f_1$, then $\Pol_n f_2 \prec \Pol_n f_1$.
        \end{cor}
        \begin{proof}
        By definition, $f_2(x)y+f_1(x)\in \Ga[x,y]$. By Lemma \ref{NRT-redundant}, its Fang-Ma-G{\aa}rding component is NRT. By Theorem \ref{kappa-polarization}, if we choose $\kappa=(n,1)$, then the $\kappa$-polarization 
        \[
        g_2(\mathbf{x})y+g_1(\mathbf{x})\in \Ga[\mathbf{x},y],
        \]
        where $g_2=\Pol_n f_2$ and $g_1=\Pol_n f_1$. Thus, $g_2\prec g_1$ by definition.
        \end{proof}
    The main reason for us to use the proper position condition is the following:
    \begin{thm}[Theorem 6.14 of \cite{FangMaGarding}]
\label{PositiveDerivative}
Let $g_2$ and $g_1$ be two nontrivial multi-affine Fang--Ma--G{\aa}rding polynomials. Then the following are equivalent:
\begin{enumerate}
    \item $g_2\prec g_1$;
    \item $\Upsilon_{g_2} \cap\{g_1>0\}=\Upsilon_{g_1}$;
    \item For every $i=1,2,\ldots,n$, $g_1\partial_i g_2-g_2\partial_i g_1 \le 0$ on $\Upsilon_{g_2}$.
\end{enumerate}
\end{thm}

    For univariate right-Noetherian polynomials, proper position can be characterized as follows.

\begin{thm}[Theorem 10.12 of \cite{FangMaGarding}]
\label{PositiveUnivariable}
Let $f_2, f_1\in \Ga_1$ be univariate monic polynomials of the same degree $d$. Then $f_2\prec f_1$ if and only if, for every $k=0,1,\ldots,d-1$, $r(f_2^{(k)}) \le r(f_1^{(k)})$ and $f_1^{(k)}(x)\le 0$ for all $x \in [r(f_2^{(k)}),r(f_1^{(k)}))$.
\end{thm}

We also recall the $\Upsilon$-dominance from \cite{LinInverseSigmaConvexity}. The extension to Fang-Ma-Garding polynomials is done by \cite{FangMaGarding}.

\begin{defn}[Definition 2.8 of \cite{LinInverseSigmaConvexity}]
\label{DefUpsilonDominance}
Let $f_2, f_1\in \Ga_1$ be univariate monic polynomials of degree $n$.  If $r(f_2^{(k)}) \leq r(f_1^{(k)})$ for every  $k=0, \cdots, n-1$, then we say $f_2 \lessdot f_1$.
\end{defn}

\begin{thm}[Theorem 2.4 of  \cite{LinInverseSigmaConvexity}]
\label{ThmUpsilonDominance}
Let $f_2, f_1$ be as before. Then $f_2\lessdot f_1$ if and only if $\Upsilon_{g_1} \subset \Upsilon_{g_2}$, where $g_i$ is the polarization of $f_i$ on $\mathbb{R}^n$.    
\end{thm}
       
        We now define the cone condition.
        \begin{defn}
            Let $f(x) = \sum_{k=0}^{d} c_k \binom{n}{k} x^k$, with $c_d>0$, be a right-Noetherian polynomial of degree $d\le n$, and let $g(\boldsymbol{x})$ be its polarization on $\mathbb{R}^n$. We define the $\Upsilon^k$ cone by
            \[\Upsilon^k_g=\bigcap_{|\alpha|=k} \Upsilon_{\partial^\alpha g}.\]
            If \(x \in \Upsilon^1_g\), we say that $x$ satisfies the cone condition. If $\omega$ is a given K\"ahler form, then we say that $\chi$ satisfies the cone condition if the eigenvalues $\lambda\!\left[\omega^{-1}\chi\right]$ satisfy the cone condition. When $d=n$, we also say that $\chi$ is a $C$-subsolution if $\chi$ satisfies the cone condition.
            \label{DefnConeCondition}
        \end{defn}

        By Section~9.2 of \cite{FangMaGarding}, the $\Upsilon^k$ cone defined here is the same as that in Definition~2.6 of \cite{LinInverseSigmaConvexity}. So we can use the following convexity property.

        \begin{prop}
        Let $f(x) = \sum_{k=0}^{n}c_k \binom{n}{k} x^k$, with $c_n=1$, be a right-Noetherian polynomial, and let $g(\boldsymbol{x})$ be its polarization on $\mathbb{R}^n$. Then $\Upsilon^k_g$ is convex for all $k=0, 1, ..., n-1$.
        Moreover, the set of Hermitian-matrices $A$ such that its eigenvalues $\lambda(A)\in \Upsilon^k_g$ is also convex. 
        \label{Convexity}
        \end{prop}
        \begin{proof}
        By Theorem 3.1 of \cite{LinInverseSigmaConvexity}, $\Upsilon_g$ is convex. Using Definition 2.6 of \cite{LinInverseSigmaConvexity} and its equivalent characterization, $\Upsilon^k_g$ is also convex. Then we apply the Ky Fan--Lidskii inequality and Rado's corollary of the Hardy--Littlewood--P\'{o}lya theorem to see that the corresponding subset of Hermitian-matrices is also convex. For example, see Corollary 2.20 of \cite{ChenNieXuNumericalCriterion} for more details.
        \end{proof}

    Now let us recall the following useful lemma proved by Lin in \cite{LinInverseSigmaConvexity} and used in \cite{LinInverseSigmaSolvability}.
    \begin{lem}
    Let $f(x) = \sum_{k=0}^{n}c_k \binom{n}{k} x^k$, with $c_n=1$, be a strictly right-Noetherian polynomial of degree $n$, and let $g(\boldsymbol{x})$ be its polarization on $\mathbb{R}^n$. Then $\Upsilon^1_g \cap \{g=0\} = \partial\Upsilon_g$.
    \label{rewrite-equation}
    \end{lem}
    \begin{proof}
    By Theorem 1.2 and Definition 2.7 of \cite{LinInverseSigmaConvexity}, $\partial\Upsilon_g\subset \Upsilon^1_g \cap \{g=0\}$. Conversely, if $\boldsymbol{x}\in \Upsilon^1_g \cap \{g=0\}$, then $\boldsymbol{x}+(s,s,...,s) \in \Upsilon^1_g$ since $\Upsilon^1_g$ passes PRT. So $g(\boldsymbol{x}+(s,s,...,s))>0$ for all $s>0$. This implies that $\boldsymbol{x}+(s,s,...,s)\in \Upsilon_g$ for all $s>0$. We also know that there exists $\delta>0$ such that $g(\boldsymbol{x}+(s,s,...,s))<0$ for all $-\delta<s<0$. This implies that $\boldsymbol{x}\in\partial\Upsilon_g$.
    \end{proof}
        
    Now we start to study differential forms on a manifold with complex dimension $n$. Assume that $\omega$ is a given K\"ahler form, $\chi$ is a real closed $(1,1)$ form, $c_n=1$, and 
    \[
	   f(x) = \sum_{k=0}^{n} c_k \binom{n}{k} x^k, \quad g(\lambda) = \sum_{k=0}^{n} c_k \sigma_k(\lambda), \quad	g(\chi) = \sum_{k=0}^{n} c_k\binom{n}{k} \chi^k \wedge \omega^{n-k}.
        \]
    For any $l$, we define the formal derivative with respect to $\chi$
    \[
    g^{(l)}(\chi)=\sum_{k=l}^{n} c_k\binom{n}{k} \frac{k!}{(k-l)!} \chi^{k-l} \wedge \omega^{n-k} = \frac{n!}{(n-l)!}\sum_{k=l}^{n} c_k\binom{n-l}{k-l} \chi^{k-l} \wedge \omega^{n-k}.
    \]

    Recall the following classical definition from Demailly's book:
    \begin{defn}[Definition III.1.1 of \cite{DemaillyBook}]
    Let $X$ be a complex manifold of dimension $n$, and let $u$ be a real $(p,p)$-form on $X$. We say that $u$ is \emph{strongly positive} if, at every point $x\in X$, it can be written as
    \[
    u_x=\sum_s c_s\,\sqrt{-1}\alpha_{s,1}\wedge\overline{\alpha_{s,1}}\wedge\cdots\wedge\sqrt{-1}\alpha_{s,p}\wedge\overline{\alpha_{s,p}},
    \]
    where $c_s\ge 0$ and $\alpha_{s,j}\in T^{1,0*}_xX$. We say that $u$ is \emph{positive} if, for every $\beta_1,\ldots,\beta_{n-p}\in T^{1,0*}_xX$,
    \[   u_x\wedge\sqrt{-1}\beta_1\wedge\overline{\beta_1}\wedge\cdots\wedge\sqrt{-1}\beta_{n-p}\wedge\overline{\beta_{n-p}}
    \]
    is a nonnegative $(n,n)$-form.
    \end{defn}

    Motivated by this, if \[   u_x\wedge\sqrt{-1}\beta_1\wedge\overline{\beta_1}\wedge\cdots\wedge\sqrt{-1}\beta_{n-p}\wedge\overline{\beta_{n-p}}>0
    \]
    for all \[\sqrt{-1}\beta_1\wedge\overline{\beta_1}\wedge\cdots\wedge\sqrt{-1}\beta_{n-p}\wedge\overline{\beta_{n-p}}\neq 0,\]
    then we say that $u$ is strictly positive.
    
We need the following lemma, which characterizes the $\Upsilon^k$-cone in Definition \ref{DefnConeCondition}:
\begin{lem}
\label{positivity}
If $f$ is right-Noetherian, then the eigenvalue $\lambda[\omega^{-1}\chi]\in \Upsilon_g$ if and only if $g^{(l)}(\chi)$ is a strictly positive $(n-l,n-l)$ form for all $l=0,1,...,n-1$. More generally, the eigenvalue $\lambda[\omega^{-1}\chi]\in \Upsilon^p_g$ if and only if $g^{(l)}(\chi)$ is a strictly positive $(n-l,n-l)$ form for all $l=p,...,n-1$.
\end{lem}

\begin{proof}
Fix a point \(x\in X\). Choose an \(\omega\)-unitary
holomorphic coframe at \(x\) in which
\[
\omega
=
\sum_{j=1}^{n} \sqrt{-1} \beta_j \wedge \bar\beta_j,
\qquad
\chi
=
\sum_{j=1}^{n}\lambda_j \sqrt{-1} \beta_j \wedge \bar\beta_j,
\qquad
\beta_j
=
dz^j.
\]

For a subset \(I=\{i_1,\ldots,i_p\}\subset\{1,\ldots,n\}\),
write
\[
\beta_I
=\sqrt{-1} \beta_{i_1} \wedge \bar\beta_{i_1}\wedge\cdots \wedge \sqrt{-1}\beta_{i_p} \wedge \bar\beta_{i_p}.
\]

Then
    \[
    g^{(l)}(\chi)=\frac{n!}{(n-l)!}\sum_{k=l}^{n} c_k\binom{n-l}{k-l} \chi^{k-l} \wedge \omega^{n-k}=n!
\sum_{|I|=n-l}
\left[
\sum_{k=l}^{n}
c_k
\sigma_{k-l}(\lambda_i, i \in I)
\right]\beta_I.
    \]

Note that
\[
\partial_{I^c}\sigma_k(\lambda) = \sigma_{k-l}(\lambda_i, i \in I).
\]
Consequently,
\[
g^{(l)}(\chi)
=n!
\sum_{|I|=n-l}
(\partial_{I^c} g)\beta_I.
\]

Thus we can test it against $\beta_{I^c}$ and translate the positivity to $\partial_{I^c} g>0$. Remark that Demailly has proved that any strongly positive form is positive in his book \cite{DemaillyBook}, and we have used the strict inequality version of that here. Finally, the proof follows from Lemma 2.5 of \cite{LinInverseSigmaConvexity}.

\end{proof}

By Lemma \ref{positivity}, if $n\ge 2$ and $c_{n-1}=0$, $\lambda(\omega^{-1}\chi)\in \Upsilon_g^1$ implies that $\chi$ is K\"ahler. For the polynomial $f_2$, we can achieve this condition by studying the polynomials $\hat f(x)=f(x - c^{(2)}_{n-1})$ and the class $\hat\chi=\chi+c^{(2)}_{n-1}\omega$ instead. Then $\hat c_n$ is still $1$ but $\hat c_{n-1}=0$.

As a corollary, we prove the following restriction property:
\begin{cor}
\label{restriction}
Suppose that 
\[
\tilde f(x)=\sum_{k=0}^{n+1} \tilde c_k \binom{n+1}{k} x^k, \quad\text{with }\tilde c_{n+1}=1,
\]
is a right-Noetherian polynomial and $\tilde g(\lambda)$ is the polarization of $\tilde f$. Consider 
\[
f(x)=\frac{1}{n+1}\tilde f'(x)=\sum_{k=0}^{n} c_k \binom{n}{k} x^k, \quad c_k=\tilde c_{k+1}
\]
and its polarization $g(\lambda)$. Then if $\omega$ is K\"ahler on the $(n+1)$-dimensional manifold $\mathcal{X}=X\times\mathcal{A}$ and $\lambda(\omega^{-1}\chi) \in \Upsilon^1_{\tilde g}$, then $\lambda( (\omega|_{X_{z_0}})^{-1}\chi|_{X_{z_0}}) \in \Upsilon_{g}$ for all $z_0\in \mathcal{A}$.
\end{cor}
\begin{proof}
By Lemma \ref{positivity}, $\tilde g^{(l)}(\chi)$ is strictly positive for all $l=1,2,...,n$. So 
\[
g^{(l)}(\chi)=\frac{1}{n+1}\tilde g^{(l+1)}(\chi)
\]
is strictly positive for all $l=0,1,...,n-1$. The restriction of any strictly positive form to the subspace is still strictly positive. So $g^{(l)}(\chi|_{X_{z_{0}}})$ is strictly positive for all $l=0,1,...,n-1$. Then we see that $\lambda( (\omega|_{X_{z_{0}}})^{-1}\chi|_{X_{z_{0}}}) \in \Upsilon_{g}$ by Lemma \ref{positivity}.
\end{proof}

In the converse direction, we can prove the following extension property:
\begin{prop}
\label{extension}
Suppose that 
\[
\tilde f(x)=\sum_{k=0}^{n+1} \tilde c_k \binom{n+1}{k} x^k, \quad\text{with }\tilde c_{n+1}=1\text{ and }\tilde c_{n}=0,
\]
is a right-Noetherian polynomial and $\tilde g(\lambda)$ is the polarization of $\tilde f$. Consider 
\[
f(x)=\frac{1}{n+1}\tilde f'(x)=\sum_{k=0}^{n} c_k \binom{n}{k} x^k, \quad c_k=\tilde c_{k+1}
\]
and its polarization $g(\lambda)$.
If $(\lambda_1, ..., \lambda_n) \in \Upsilon_g$, then the following are equivalent:
\begin{enumerate}
\item $(\lambda_1, ..., \lambda_n, s) \in \Upsilon_{\tilde g}$;
\item $\tilde g (\lambda_1, ..., \lambda_n, s) > 0$;
\item $s>-\frac{\tilde g (\lambda_1, ..., \lambda_n, 0)}{g (\lambda_1, ..., \lambda_n)}$.
\end{enumerate}
Moreover,
\[
\tilde g (\lambda_1, ..., \lambda_n, 0) = \sum_{k=0}^{n} \tilde c_k \sigma_k(\lambda_1, ..., \lambda_n) \le 0, \forall (\lambda_1, ..., \lambda_n) \in \Upsilon_g.
\]
\end{prop}
\begin{proof}
It is trivial to see that $(1)\Rightarrow (2)$ because $\Upsilon_{\tilde g}$ is a connected component of $\{\tilde g>0\}$. To get $(2)\Longleftrightarrow (3)$, note that
\[\tilde g(\lambda_1, ..., \lambda_n, s)=\sum_{k=0}^{n+1} \tilde c_k \sigma_k(\lambda_1, ..., \lambda_n, s).
\]
So
\[\partial_s \tilde g(\lambda_1, ..., \lambda_n, s)=\sum_{k=1}^{n+1} \tilde c_k \sigma_{k-1}(\lambda_1, ..., \lambda_n)=g(\lambda_1, ..., \lambda_n).
\]
Using the fact that $\tilde g$ is multiaffine, we see that 
\[
\tilde g (\lambda_1, ..., \lambda_n, s)
=\tilde g (\lambda_1, ..., \lambda_n, 0) + s \partial_s \tilde g (\lambda_1, ..., \lambda_n, 0)
=\tilde g (\lambda_1, ..., \lambda_n, 0) + s g (\lambda_1, ..., \lambda_n).
\]
Note that $g (\lambda_1, ..., \lambda_n)>0$ on $\Upsilon_g$. This proves $(2)\Longleftrightarrow (3)$.
Now we use the proof of Lemma 2.4 of \cite{LinInverseSigmaConvexity} to finish $(3)\Rightarrow (1)$. For the convenience of readers, we write the proof here. The function $h(\lambda_1, ..., \lambda_n)=-\frac{\tilde g (\lambda_1, ..., \lambda_n, 0)}{g (\lambda_1, ..., \lambda_n)}$ is a continuous function on $\Upsilon_g$. Pick any $(\mu_1,...,\mu_n, \mu_{n+1})\in \Upsilon_{\tilde g}$. Then $(\mu_1,...,\mu_n)\in \Upsilon_g$, and $\mu_{n+1}>h(\mu_1,...\mu_n)$. Find a continuous path $(\lambda_1(t),...\lambda_n(t))$ inside $\Upsilon_g$ connecting $(\lambda_1, ..., \lambda_n)$ and $(\mu_1,...,\mu_n)$. Let $H$ be a number larger than $\mu_{n+1}$ and the upper bound of $h$ on this path. Then $(\lambda_1(t),...\lambda_n(t),H)$ is a path inside $\{\tilde g>0\}$. The path $(\mu_1,...\mu_n,s), s\in[\mu_{n+1},H]$ and the path $(\lambda_1,...\lambda_n,s), s>h(\lambda_1,...\lambda_n)$ are also paths in $\{\tilde g>0\}$. So all of them lie in $\Upsilon_{\tilde g}$.

Finally, if $\tilde c_{n}=c_{n-1}=0$, then $\Upsilon_{\tilde g}\subset\Upsilon^{n}_{\tilde g}=\Gamma_{n+1}^+$. However, $(\lambda_1, ..., \lambda_n, 0) \not\in \Gamma_{n+1}^+$.
This implies  
\[
\tilde g (\lambda_1, ..., \lambda_n, 0) = \sum_{k=0}^{n} \tilde c_k \sigma_k(\lambda_1, ..., \lambda_n) \le 0
\]
by the previous equivalence in this proposition.
\end{proof}

\section{The Geodesic Equation  }

For $i=1,2$, set $c_n^{(i)}=1$ and $c_{n-1}^{(2)}=0$. We define
	   \[
	   f_i(x) = \sum_{k=0}^{n} c_k^{(i)} \binom{n}{k} x^k, \quad
	   g_i(\lambda) = \sum_{k=0}^{n} c_k^{(i)} \sigma_k(\lambda),
	   \]
	   \[
		   g_i(\chi_\varphi) = \sum_{k=0}^{n} c_k^{(i)}\binom{n}{k} \chi_\varphi^k \wedge \omega^{n-k},\quad i=1,2.
		   \]
		   Assume that $f_1, f_2$ are right-Noetherian polynomials such that $f_2\prec f_1$.
  
        Let $\mathcal{H} = \{ \varphi \in C^{\infty}(X,\mathbb{R})\mid \lambda(\omega^{-1}\chi_\varphi)\in \Upsilon_{g_2} \}$. The tangent space at $\varphi\in\mathcal{H}$ is $T_{\varphi}\mathcal{H}=C^{\infty}(X,\mathbb{R})$. Consider the following Riemannian metric:
        \[ \langle \Psi_1, \Psi_2 \rangle_{\varphi} = \int_{X} \Psi_1 \Psi_2 g_2(\chi_\varphi)
        =\int_{X} \Psi_1 \Psi_2 \sum_{k=0}^{n} \binom{n}{k} c_k^{(2)} \chi_{\varphi}^{k}\wedge  \omega^{n-k}.\]
        Let $\varphi(t)$ be a smooth curve in $\mathcal{H}$ with $\varphi(0)=\varphi_{0}$ and $\varphi(1)=\varphi_{1}$. Let us also consider the family of curves $\varphi(t,s)$ in $\mathcal{H}$ such that $\varphi(0,s)=\varphi_{0},\varphi(1,s)=\varphi_{1}$ and $\varphi(t,0)=\varphi(t)$. 
        Consider the energy
        \[E(s) = \int_{0}^{1} \int_{X} \left( \frac{d\varphi}{dt} \right)^{2}\sum_{k=0}^{n} \binom{n}{k} c_k^{(2)} \chi_{\varphi}^{k}\wedge  \omega^{n-k}dt.\]
        We now differentiate with respect to $s$.
        
        \begin{align*}
            \frac{dE}{ds}\Big|_{s=0} & = \int_0^1 \int_{X} 2\frac{d\varphi}{dt}\frac{d}{dt}\left(\frac{d}{ds}\Big|_{s=0}\varphi\right)\sum_{k=0}^{n} \binom{n}{k} c_k^{(2)} \chi_{\varphi}^{k}\wedge  \omega^{n-k}dt\\
            & +\int_{0}^{1}\int_{X}\big(\frac{d\varphi}{dt}\big)^{2}\sum_{k=1}^{n}\binom{n}{k}c_k^{(2)}k\chi_{\varphi}^{k-1}\wedge\omega^{n-k}\wedge\sqrt{-1}\partial\bar{\partial}\left(\frac{d}{ds}\Big|_{s=0}\varphi\right)dt
        \end{align*}
        
        Integrating the second term by parts, we obtain
        \begin{align*}
            & = \int_0^1 \int_{X} 2\frac{d\varphi}{dt}\frac{d}{dt}\left(\frac{d}{ds}\Big|_{s=0}\varphi\right)\sum_{k=0}^{n} \binom{n}{k} c_k^{(2)} \chi_{\varphi}^{k}\wedge  \omega^{n-k}dt\\
            &+2\int_{0}^{1}\int_{X}(\frac{d}{ds}\Big|_{s=0}\varphi)\sqrt{-1}\partial(\frac{d\varphi}{dt})\wedge \bar{\partial}(\frac{d\varphi}{dt})\wedge\sum_{k=1}^{n}\binom{n}{k}c_k^{(2)}k\chi_{\varphi}^{k-1}\wedge\omega^{n-k}dt\\
            &+2\int_{0}^{1}\int_{X}(\frac{d}{ds}\Big|_{s=0}\varphi)(\frac{d\varphi}{dt}) \sqrt{-1}\partial\bar{\partial}(\frac{d\varphi}{dt})\wedge\sum_{k=1}^{n}\binom{n}{k}c_k^{(2)}k\chi_{\varphi}^{k-1}\wedge\omega^{n-k}dt
        \end{align*}
        Integrating the first term by parts in time, we obtain
        \begin{align*}
            &\int_0^1 \int_{X} 2\frac{d\varphi}{dt}\frac{d}{dt}\left(\frac{d}{ds}\Big|_{s=0}\varphi\right)\sum_{k=0}^{n} \binom{n}{k} c_k^{(2)} \chi_{\varphi}^{k}\wedge  \omega^{n-k}dt\\
            &=-\int_{0}^{1}\int_{X}2\frac{d^{2}\varphi}{dt^{2}}\left(\frac{d}{ds}\Bigg|_{s=0}\varphi\right)\sum_{k=0}^{n} \binom{n}{k} c_k^{(2)} \chi_{\varphi}^{k}\wedge  \omega^{n-k}dt\\
            &-\int_{0}^{1}\int_{X}2\frac{d\varphi}{dt}\left(\frac{d}{ds}\Bigg|_{s=0}\varphi\right)\sum_{k=1}^{n}\binom{n}{k}c_k^{(2)}k\chi_{\varphi}^{k-1}\wedge\omega^{n-k}\wedge\sqrt{-1}\partial\bar{\partial}\left(\frac{d\varphi}{dt}\right)dt.
        \end{align*}
        Therefore,
        \begin{align*}
            \frac{dE}{ds}\Bigg|_{s=0} &=2\int_{0}^{1}\int_{X}(\frac{d}{ds}\Big|_{s=0}\varphi)\sqrt{-1}\partial(\frac{d\varphi}{dt})\wedge \bar{\partial}(\frac{d\varphi}{dt})\wedge\sum_{k=1}^{n}\binom{n}{k}c_k^{(2)}k\chi_{\varphi}^{k-1}\wedge\omega^{n-k}dt\\
            &-\int_{0}^{1}\int_{X}2\frac{d^{2}\varphi}{dt^{2}}\left(\frac{d}{ds}\Bigg|_{s=0}\varphi\right)\sum_{k=0}^{n} \binom{n}{k} c_k^{(2)} \chi_{\varphi}^{k}\wedge  \omega^{n-k}dt
        \end{align*}

        Thus, we obtain the geodesic equation
        \begin{equation}
                \frac{d^2\varphi}{dt^2} \sum_{k=0}^{n}c_k^{(2)}\binom{n}{k}\chi_{\varphi}^k\wedge\omega^{n-k} 
                -\sqrt{-1}\partial\frac{d\varphi}{dt}\wedge  \bar{\partial}(\frac{d\varphi}{dt})\wedge 
                \sum_{k=1}^{n}kc_k^{(2)}\binom{n}{k}\chi_{\varphi}^{k-1}\wedge\omega^{n-k} =0.
            \label{GeodesicEq2}        
        \end{equation}
    
        \subsection{Geodesic equation in one dimension higher}
            
            We now write this geodesic equation in $n+1$ dimensions. For this purpose, consider $\mathcal{X}=X\times\mathcal{A}$, where
            \[
            \mathcal{A}=\{z_0\in\mathbb{C}\mid 1<\lvert z_0\rvert<e\},
            \qquad t=\log\lvert z_0\rvert\in(0,1).
            \]
            \begin{lem}
                For $\varphi_{0},\varphi_{1}\in \mathcal{H}$, let $\varphi:\mathcal{X}\rightarrow\mathbb{R}$ be $S^{1}$-invariant, that is, $\varphi(x,z_0)=\varphi(x,\lvert z_0\rvert)$, and set
                \[
                \varphi_t(x)=\varphi(x,z_0),\qquad t=\log\lvert z_0\rvert.
                \]
                Then $\varphi_t$ solves the geodesic equation with $\varphi_t=\varphi_0$ when $t=0$ and $\varphi_t=\varphi_1$ when $t=1$ if and only if
                \[\sum_{k=1}^{n+1}\binom{n+1}{k}\tilde{c}_{k}(\pi_{X}^{*}\chi+\sqrt{-1}D\bar{D}\varphi)^{k}\wedge(\pi_{X}^{*}\omega)^{n+1-k}=0, \quad \tilde c_k=c_{k-1}^{(2)}.\]
                Equivalently, the boundary conditions are
                \[
                \varphi|_{\{\lvert z_0\rvert=1\}}=\varphi_0,
                \qquad
                \varphi|_{\{\lvert z_0\rvert=e\}}=\varphi_1.
                \]
                Here $\sqrt{-1}D\bar{D}$ is the $\sqrt{-1}\partial\bar{\partial}$ operator on the $(n+1)$-dimensional manifold $\mathcal{X}$.
            \end{lem}

            \begin{proof}
                Since $t=\log\lvert z_0\rvert$, we have the identities
                \[\partial_{z_0}\varphi=\frac{1}{2z_0}\frac{d\varphi_t}{dt},\qquad \partial_{\bar z_0}\varphi=\frac{1}{2\bar z_0}\frac{d\varphi_t}{dt},\qquad \partial_{z_0}\partial_{\bar z_0}\varphi=\frac{1}{4\lvert z_0\rvert^{2}}\frac{d^{2}\varphi_t}{dt^{2}}.\]
                Observe that $\pi_{X}^{*}\chi+\sqrt{-1}D\bar{D}\varphi=\chi_{\varphi_t}+\alpha$ where 
                \[
\alpha=(\sqrt{-1}\partial_{z_0}\partial_{\bar z_0}\varphi)dz_0\wedge d\bar z_0+\sqrt{-1}\partial_{X}\partial_{\bar z_0}\varphi\wedge d\bar z_0+\sqrt{-1}dz_0\wedge \partial_{z_0}\bar{\partial}_{X}\varphi.
                \]
                We also have $\alpha^{j}=0$ for $j\ge 3$, $\chi_{\varphi_t}^k\wedge (\pi_{X}^{*}\omega)^{n+1-k}=0$, and 
                \[
                \alpha^{2}=-(2\sqrt{-1}\partial_{X}\partial_{\bar z_0}\varphi\wedge \sqrt{-1}\partial_{z_0}\bar{\partial}_{X}\varphi)\wedge dz_0\wedge d\bar z_0.
                \]
                We now compute
                \begin{align*}
                    &\left(\pi_{X}^{*}\chi+\sqrt{-1}D\bar{D}\varphi\right)^{k}\wedge (\pi_{X}^{*}\omega)^{n+1-k}\\
                    &=(\chi_{\varphi_t}+\alpha)^{k}\wedge (\pi_{X}^{*}\omega)^{n+1-k}\\
                    &=\left(k\chi_{\varphi_t}^{k-1}\wedge \alpha+\frac{k(k-1)}{2}\chi_{\varphi_t}^{k-2}\wedge \alpha^{2}\right)\wedge (\pi_{X}^{*}\omega)^{n+1-k}\\
                    &=\left( k\chi_{\varphi_t}^{k-1} \partial_{z_0}\partial_{\bar z_0}\varphi-k(k-1)\chi_{\varphi_t}^{k-2}\wedge \sqrt{-1}\partial_{X}\partial_{\bar z_0}\varphi\wedge \partial_{z_0}\bar{\partial}_{X}\varphi\right)\wedge \sqrt{-1} dz_0\wedge d\bar z_0\wedge (\pi_{X}^{*}\omega)^{n+1-k} \\
                    &=k\frac{\sqrt{-1}dz_0\wedge d\bar z_0}{4\lvert z_0\rvert^{2}}\left(\frac{d^{2}\varphi_t}{dt^{2}}\chi_{\varphi_t}^{k-1}-(k-1)\chi_{\varphi_t}^{k-2}\wedge \sqrt{-1}\partial_{X}(\frac{d\varphi_t}{dt})\wedge \bar{\partial}_{X}(\frac{d\varphi_t}{dt})\right) \wedge (\pi_{X}^{*}\omega)^{n+1-k}
                \end{align*}
                Combining these identities with the combinatorial equation
                \[\binom{n+1}{k}k=(n+1)\binom{n}{k-1},\] we obtain
                \begin{equation}
                \begin{split}
                \label{higher-dimensional-geodesic}
                    &\frac{1}{n+1}\sum_{k=1}^{n+1}\binom{n+1}{k}\tilde{c}_{k}(\pi_{X}^{*}\chi+\sqrt{-1}D\bar{D}\varphi)^{k}\wedge(\pi_{X}^{*}\omega)^{n+1-k}\\
                    &=\frac{\sqrt{-1}dz_0\wedge d\bar z_0}{4\lvert z_0\rvert^{2}}\frac{d^{2}\varphi_t}{dt^{2}}\left(\sum_{k=1}^{n+1}\binom{n}{k-1}\tilde{c}_{k}\chi_{\varphi_t}^{k-1}\wedge(\pi_{X}^{*}\omega)^{n+1-k}\right)\\
                    &-\frac{\sqrt{-1}dz_0\wedge d\bar z_0}{4\lvert z_0\rvert^{2}}\sqrt{-1}\partial_{X}(\frac{d\varphi_t}{dt})\wedge \bar{\partial}_{X}(\frac{d\varphi_t}{dt})\left(\sum_{k=2}^{n+1}\binom{n}{k-1}\tilde{c}_{k}(k-1)\chi_{\varphi_t}^{k-2}\wedge (\pi_{X}^{*}\omega)^{n+1-k}\right).
                \end{split}
                \end{equation}
                If $\tilde c_k=c_{k-1}^{(2)}$ for $1\le k\le n+1$, then we obtain
                $\frac{\sqrt{-1}dz_0\wedge d\bar z_0}{4\lvert z_0\rvert^{2}}\times \text{(geodesic equation)}$.
            \end{proof}

             \subsection{The \texorpdfstring{$\epsilon$}{Epsilon}-geodesic equation}

            We can't expect a smooth geodesic by the counterexample of Darvas and Lempert \cite{DarvasLempertWeakGeodesics}. So in this subsection, we study the $(n+1)$-dimensional $\epsilon$-geodesic equation
            \[
            \sum_{k=0}^{n+1} \tilde{c}_k \binom{n+1}{k} (\pi^*_{X}\chi + \sqrt{-1}D\bar{D}\varphi )^{k}\wedge (\pi_{X}^*\omega + \epsilon^{2}\sqrt{-1}dz_0\wedge  d\bar z_0)^{n+1-k} =0,
            \]
            with the assumption that $\pi_{X}^{*}\chi+\sqrt{-1}D\bar{D}\varphi \in \Upsilon^1_{\tilde g}$. By Corollary \ref{restriction}, any such solution is a path inside $\mathcal{H}$.
            By Lemma~\ref{IntegralUpsilonStable}, 
            the constant $\tilde{c}_0$ will be chosen sufficiently negative so that the polynomial associated to the equation is strictly right-Noetherian.

            	\begin{lem}
    \label{geodesic lemma}
            The $\epsilon$-geodesic $\varphi_{\epsilon}$ satisfies
		\begin{align*}
			&\frac{d^{2}\varphi_{\epsilon}}{dt^{2}}g_{2}(\chi_{\varphi_{\epsilon}})-\sqrt{-1}\partial(\frac{d\varphi_{\epsilon}}{dt})\wedge\bar{\partial}(\frac{d\varphi_{\epsilon}}{dt})\wedge g_{2}'(\chi_{\varphi_{\epsilon}})\\
			&=-4\epsilon^2e^{2t}\left(\sum_{k=0}^{n}\tilde{c}_{k}\binom{n}{k}\big(\pi_{X}^{*}\chi+\sqrt{-1}\partial\bar{\partial}\varphi_{\epsilon}\big)^{k}\wedge(\pi_{X}^{*}\omega)^{n-k}\right) \ge 0.
		\end{align*}
	\end{lem}
	\begin{proof}
		 Separating the $\epsilon^{2}$ term, and using $(\pi_{X}^{*}\omega)^{n+1}=0$, we obtain
\begin{align*}
&\sum_{k=1}^{n+1}\tilde{c}_{k}\binom{n+1}{k}(\pi_{X}^{*}\chi+\sqrt{-1}D\bar{D}\varphi_\epsilon)^{k}\wedge(\pi_{X}^{*}\omega)^{n+1-k}\\
	&=-\epsilon^{2}\sqrt{-1}dz_0\wedge d\bar z_0\wedge\left(\sum_{k=0}^{n}\tilde{c}_{k}(n+1-k)\binom{n+1}{k}\big(\pi_{X}^{*}\chi+\sqrt{-1}\partial\bar{\partial}\varphi_{\epsilon}\big)^{k}\wedge(\pi_{X}^{*}\omega)^{n-k}\right).
\end{align*}
By \eqref{higher-dimensional-geodesic}, we know that
\begin{align*}
	&\sum_{k=1}^{n+1}\tilde{c}_{k}\binom{n+1}{k}(\pi_{X}^{*}\chi+\sqrt{-1}D\bar{D}\varphi_\epsilon)^{k}\wedge(\pi_{X}^{*}\omega)^{n+1-k}\\
		&=(n+1)\frac{\sqrt{-1}dz_0\wedge d\bar z_0}{4\lvert z_0\rvert^{2}}\wedge \left(\frac{d^{2}\varphi_\epsilon}{dt^{2}}g_{2}(\chi_{\varphi_\epsilon})-\sqrt{-1}\partial(\frac{d\varphi_\epsilon}{dt})\wedge\bar{\partial}(\frac{d\varphi_\epsilon}{dt})\wedge g_{2}'(\chi_{\varphi_\epsilon})\right).
\end{align*}
Hence,
\begin{align*}
	&\frac{d^{2}\varphi_{\epsilon}}{dt^{2}}g_{2}(\chi_{\varphi_{\epsilon}})-\sqrt{-1}\partial(\frac{d\varphi_{\epsilon}}{dt})\wedge\bar{\partial}(\frac{d\varphi_{\epsilon}}{dt})\wedge g_{2}'(\chi_{\varphi_{\epsilon}})\\
	&=-\frac{4\epsilon^2e^{2t}}{n+1}\left(\sum_{k=0}^{n}\tilde{c}_{k}(n+1-k)\binom{n+1}{k}\big(\pi_{X}^{*}\chi+\sqrt{-1}\partial\bar{\partial}\varphi_{\epsilon}\big)^{k}\wedge(\pi_{X}^{*}\omega)^{n-k}\right)\\
    &=-4\epsilon^2e^{2t}\left(\sum_{k=0}^{n}\tilde{c}_{k}\binom{n}{k}\big(\pi_{X}^{*}\chi+\sqrt{-1}\partial\bar{\partial}\varphi_{\epsilon}\big)^{k}\wedge(\pi_{X}^{*}\omega)^{n-k}\right).
\end{align*}
Finally, this term is non-negative on $\mathcal{H}$ by Proposition \ref{extension}.
	\end{proof}

\section{Path and Zeroth Order Estimate}
\label{SecContinuityPath}

The $\epsilon$-geodesic equation is the following equation
\begin{equation}
\label{modified equation}
    \begin{cases}
        \sum_{k=0}^{n+1} \tilde{c}_k \binom{n+1}{k} (\pi^*_{X}\chi + \sqrt{-1}D\bar{D}\varphi )^{k}\wedge (\pi_{X}^*\omega + \epsilon^{2}\sqrt{-1}dz_0\wedge  d\bar z_0)^{n+1-k} =0, \\
        \pi_{X}^{*}\chi+\sqrt{-1}D\bar{D}\varphi \in \Upsilon^1_{\tilde g}, \quad \varphi|_{\{\lvert z_0\rvert=1\}}=\varphi_0, \quad \varphi|_{\{\lvert z_0\rvert=e\}}=\varphi_1.
    \end{cases}
\end{equation}
Recall that we always assume  $c_{n}^{(2)}=1$, and without loss of generality assume that $c_{n-1}^{(2)}=0$. Then Lemma \ref{positivity} implies that $\pi_{X}^{*}\chi+\sqrt{-1}D\bar{D}\varphi$ is K\"ahler.

This equation\eqref{modified equation} is considered on $\mathcal{X}:=X\times \mathcal{A}$ with the prescribed boundary conditions. When $\epsilon=0$, it becomes the geodesic equation. For $\epsilon>0$, we choose $\tilde{c}_0$ negative enough so that the associated polynomial is strictly right-Noetherian.

Let
\[
f_2(x) = \sum_{k=0}^{n}c_{k}^{(2)}\binom{n}{k}x^{k}, \quad g_{2}(\lambda) = \Pol_{n} (f_2) = \sum_{k=0}^{n}c_{k}^{(2)}\sigma_k(\lambda)
\]
\[
\tilde{f}_2(x) = \sum_{k=0}^{n+1}\tilde{c}_k\binom{n+1}{k}x^{k}, \quad \tilde{g}_{2}(\lambda) = \Pol_{n+1} (\tilde{f}_2) = \sum_{k=0}^{n+1}\tilde{c}_k\sigma_k(\lambda)
\]
and let $r_i = r(\tilde{f}_2^{(i)}),0\le i \le n,$ denote the largest real root of $\tilde{f}_2^{(i)}.$

Now we consider the following continuity path for \eqref{modified equation}:
\begin{equation}
\label{ContinuityPath}
    \begin{cases}
        \sum_{k=0}^{n+1} \tilde{c}_k(s) \binom{n+1}{k} (\pi^*_{X}\chi + \sqrt{-1}D\bar{D}\varphi )^{k}\wedge (\pi_{X}^*\omega + \epsilon^{2}\sqrt{-1}dz_0\wedge  d\bar z_0)^{n+1-k} =0, \\
        \pi_{X}^{*}\chi+\sqrt{-1}D\bar{D}\varphi \in \Upsilon^1_{\tilde g_{2,s}}, \quad \varphi|_{\{\lvert z_0\rvert=1\}}=\varphi_0, \quad \varphi|_{\{\lvert z_0\rvert=e\}}=\varphi_1.
    \end{cases}
\end{equation}
where 
\[
\tilde{c}_{k}(s) = s^{n+1-k}\tilde{c}_{k}, 1\le k\le n+1,  \quad \tilde{c}_{0}(s) = - \sum_{k=1}^{n+1}\tilde{c}_k(s)\binom{n+1}{k}r_0^k. 
\]
Now we denote:
\[
f_{2,s}(x) = \sum_{k=0}^{n} \tilde{c}_{k+1}(s) \binom{n}{k} x^k, \quad
g_{2,s}(\lambda) = \Pol_{n} (f_{2,s}) = \sum_{k=0}^{n} \tilde{c}_{k+1}(s) \sigma_k(\lambda),
\]
\[
\tilde{f}_{2,s}(x) = \sum_{k=0}^{n+1}\tilde{c}_k(s)\binom{n+1}{k}x^{k}, \quad \tilde{g}_{2,s}(\lambda) = \Pol_{n+1} (\tilde{f}_{2,s}) = \sum_{k=0}^{n+1} \tilde{c}_{k}(s) \sigma_k(\lambda),
\]
and $r_i(s) =  r(\tilde{f}_{2,s}^{(i)}),0\le i \le n$.
Remark that when $s=0$, $\tilde f_{2,s}=x^{n+1}-r_0^{n+1}$ is the classical $\epsilon$-geodesic studied by \cite{ChenKahlerMetrics}, and $f_{2,s}=x^n$. When $s=1$, $\tilde f_{2,s}=\tilde f_2$ and $f_{2,s}=f_2$.

\begin{lem}
    \label{PathRightNoeth}
    $\tilde{f}_{2,s}(x)$ is strictly right Noetherian. Moreover $\tilde{f}_{2,s}\lessdot \tilde{f}_2$.
\end{lem}
\begin{proof}
By assumption, $\tilde{f}_{2} = \tilde{f}_{2,1}$ is strictly right Noetherian.
So $r_0 > r_1 \geq r_2 \geq \cdots\geq r_n=0$.
A simple calculation shows that
\[
\tilde{f}_{2,s}^{(i)}(x) = s^{n+1-i}\tilde{f}_{2}^{(i)}(\frac{x}{s}), 1\le i \le n
\]
for all $s\not=0$.
So $r_{i}(s) =  sr_{i}$ for $i\geq 1$. Note that $r_{i}(s) = 0 = sr_{i}$ is also true for $s=0$.
In fact, $\tilde{c}_0(s)$ is chosen such that
\[
\tilde{f}_{2,s}(r_0) = 0.
\]
Since  $r_0 > r_1 \geq sr_1 = r_1(s)$, $r_0$ is  the largest real root of $\tilde{f}_{2,s}$ by Lagrange's mean value theorem, which means that $r_0(s) = r_0.$
Then we have 
\[
r_0(s) > r_1(s) \geq \cdots \geq r_n(s), \quad r_i \geq r_i(s).
\]
By Definition \ref{right-Noetherian} and Definition \ref{DefUpsilonDominance}, $\tilde{f}_{2,s}(x)$ is strictly right Noetherian and $\tilde{f}_{2,s}\lessdot \tilde{f}_2$.
\end{proof}

By Theorem \ref{ThmUpsilonDominance}, we have $\Upsilon_{\tilde{g}_{2}} \subset \Upsilon_{\tilde{g}_{2,s}}$. Moreover, we have $\Upsilon_{\tilde{g}_{2}}^{i} \subset \Upsilon_{\tilde{g}_{2,s}}^{i}$.
Now since $\varphi_0,\varphi_1 \in \mathcal{H}$, we have
$\varphi_0,\varphi_1 \in \mathcal{H}_s:=\{ \varphi \in C^{\infty}(X,\mathbb{R})\mid \lambda(\omega^{-1}\chi_\varphi)\in \Upsilon_{g_{2,s}} \}.$

Now we consider the set
\[
P := \{ s\in [0,1] : \text{\eqref{ContinuityPath} has a smooth solution for s}\}.
\]
When $s=0$, the existence of an $\epsilon$-geodesic of \eqref{ContinuityPath} is proved by Xiuxiong Chen in \cite{ChenKahlerMetrics}. So $P$ is non empty.
$P$ is open by the implicit function theorem since \eqref{ContinuityPath} is elliptic and the linearised operator is self-adjoint, with trivial kernel by the maximum principle. Then we only need to show the closedness of $P$.

\subsection{Zeroth order estimate}
    \label{C0section}
       
        First, we state the following lemma, which gives the $C^{0}$ estimate. 
        \begin{lem}
        \label{supsubsolutions}
            There exist an $S^1$-invariant function $\overline{\varphi}$ independent of $\epsilon$ and $s$, and an $S^1$-invariant function $\underline{\varphi}$ independent of $s$, such that if $ \varphi_{\epsilon, s} : \mathcal{X} \to \mathbb{R} $ solves \eqref{ContinuityPath} for any $s\in[0,1]$, then
            \begin{itemize}
                \item $ \underline{\varphi} \leq \varphi_{\epsilon, s} \leq \overline{\varphi}$ ;
                \item $ \underline{\varphi}, \overline{\varphi} $ satisfy the boundary conditions;
                \item $|\frac{d}{dt}\overline{\varphi}| \leq \| \varphi_0 - \varphi_1 \|_{L^{\infty}(X)}$
                and $ |\frac{d}{dt}\underline{\varphi}|\le C(\| \varphi_0 - \varphi_1 \|_{L^{\infty}(X)}+\epsilon^2)$ on $\partial\mathcal{X}$ for a constant $C$ depending only on the upper bound of $\lambda(\omega^{-1}\chi_{\varphi_i})$, $i=0,1$, and the positive lower bound of $g_2(\lambda(\omega^{-1}\chi_{\varphi_i}))$, $i=0,1$.
                \item $\pi_{X}^{*}\chi+\sqrt{-1}D\bar{D}\underline{\varphi} \in \Upsilon_{\tilde g_2}(\omega_\epsilon)\subset \Upsilon^1_{\tilde g_2}(\omega_\epsilon)$.
            \end{itemize}
        \end{lem}

        \begin{proof}

        We begin with two functions $\varphi_0,\varphi_1 \in \mathcal{H}$, which provide the boundary data for our $(n+1)$-dimensional geodesic equation.
        
        First, we will construct a supersolution. Since the solution $\pi_{X}^{*}\chi+\sqrt{-1}D\bar{D}\varphi_{\epsilon, s}$ is K\"ahler, we see that $\frac{d^2}{dt^2}\varphi_{\epsilon, s}\ge 0$ by testing it on the $\sqrt{-1}dz_0\wedge d\bar z_0$ direction.
          Let $\overline{\varphi}=(1-\log \lvert z_0\rvert)\varphi_0+\log \lvert z_0\rvert\varphi_1$. Then $\overline{\varphi}$ satisfies the boundary condition and $\varphi_{\epsilon, s}\le \overline{\varphi}$ by convexity.
          
          Then we construct the subsolution. In one higher dimension, let
            \[
            \omega_{\epsilon} = \pi_{X}^*\omega + \epsilon^2\sqrt{-1} dz_0 \wedge d\bar z_0
            \]
            be a K\"{a}hler form on $\mathcal{X}$.
            We first construct $\psi_0 = \pi_{X}^{*}\varphi_0 + \eta_0(z_0)$ such that
           \[
            \pi_{X}^{*}\chi+\sqrt{-1}D\bar{D}\psi_0 \in \Upsilon_{\tilde g}(\omega_\epsilon),
            \]
            where $\eta_0(z_0) = A_0 \epsilon^2 ( \lvert z_0\rvert^2 - 1 ) -B_0\log\lvert z_0\rvert $ for some constants $A_0$, $B_0$ such that $A_0$ is independent of $\epsilon$.
            
            Firstly, it is trivial to see that
            $\psi_0|_{\{\lvert z_0\rvert=1\}}=\varphi_0$.          
            Then we compute
            $$
            \pi_{X}^*\chi + \sqrt{-1}D\bar{D}\psi_0 = \pi_{X}^*\chi +\sqrt{-1}\partial_X\bar\partial_X\varphi_0 + \sqrt{-1}\partial_{z_0}\bar\partial_{z_0}\eta_0=\pi_{X}^*\chi_{\varphi_0} + A_0 \epsilon^2\sqrt{-1}dz_0\wedge d\bar z_0,$$
            where $\chi_{\varphi_0}=\chi+\sqrt{-1}\partial_X\bar\partial_X\varphi_0$.

            Note that the eigenvalues of $\omega_\epsilon^{-1}(\pi_{X}^*\chi + \sqrt{-1}D\bar{D}\psi_0)$ are the eigenvalues of $\omega^{-1}\chi_{\varphi_0}$ together with $A_0$. We choose $A_0>0$ depending only on the upper bound of $\lambda(\omega^{-1}\chi_{\varphi_0})$ and the positive lower bound of $g_2(\lambda(\omega^{-1}\chi_{\varphi_i}))$ such that $A_0$ is larger than the function $-\frac{\tilde g_2(\lambda_1, ..., \lambda_n, 0)}{g_2(\lambda_1,...\lambda_n)}$ in Proposition \ref{extension} for all $(\lambda_1,...\lambda_n)=\lambda(\omega^{-1}\chi_{\varphi_0})$. Then $\pi_{X}^*\chi + \sqrt{-1}D\bar{D}\psi_0\in \Upsilon_{\tilde g_2}(\omega_\epsilon)$ for all $\epsilon$.

            For $\varphi_1 \in \mathcal{H}$, we similarly define
            $$
            \psi_1 = \pi_{X}^{*}\varphi_1 + \eta_1(z_0)$$
            where $\eta_1(z_0) = A_1 \epsilon^2 ( \lvert z_0\rvert^2 - e^2 ) + B_1 \log(\frac{\lvert z_0\rvert}{e})$ for some constants $A_1,B_1$.
            The same argument proves that $\pi_{X}^*\chi + \sqrt{-1}D\bar{D}\psi_1\in \Upsilon_{\tilde g_2}(\omega_\epsilon)$ and $\psi_1 |_{\{\lvert z_0\rvert=e\}} = \varphi_1$ for $A_1$ chosen similarly.

            We now have
            $$
            \begin{cases}
                \psi_0 = \pi_{X}^{*}\varphi_0 + A_0\epsilon^2( \lvert z_0\rvert^2 - 1 ) - B_0 \log\lvert z_0\rvert \\
                \psi_1 = \pi_{X}^{*}\varphi_1 + A_1\epsilon^2( \lvert z_0\rvert^2 - e^2 ) + B_1 \log(\frac{\lvert z_0\rvert}{e})
            \end{cases}
            $$
            Moreover,
            $$
            \begin{cases}
                \psi_0(x,z_0)|_{\{\lvert z_0\rvert=e\}} = \varphi_0(x) + A_0\epsilon^2( e^2 - 1 ) - B_0  \\
                \psi_1(x,z_0)|_{\{\lvert z_0\rvert=1\}} = \varphi_1(x) + A_1\epsilon^2( 1 - e^2 ) - B_1 
            \end{cases}
            $$            
            We can choose $B_0$ and $B_1$ such that
            $$
            \begin{cases}
                \varphi_0(x) + A_0\epsilon^2( e^2 - 1 ) - B_0 \le \varphi_1(x) -\epsilon^2 = \psi_1(x,z_0)|_{\{\lvert z_0\rvert=e\}} -\epsilon^2      \\
                \varphi_1(x) + A_1\epsilon^2( 1 - e^2 ) - B_1 \le \varphi_0(x) -\epsilon^2 = \psi_0(x,z_0)|_{\{\lvert z_0\rvert=1\}} -\epsilon^2    
            \end{cases}
            $$          
            In fact, we may choose $B_0$ and $B_1$ as follows:
            $$
            \begin{cases}
                B_0 = \| \varphi_0 - \varphi_1 \|_{L^{\infty}(X)} + A_0 \epsilon^2 (e^2-1) + \epsilon^2, \\
                B_1 = \| \varphi_0 - \varphi_1 \|_{L^{\infty}(X)} + A_1\epsilon^2( 1-e^2 ) + \epsilon^2.               
            \end{cases}
            $$
            
            Now we take the regularized maximum $\underline{\varphi} = \operatorname{RegMax}\{\psi_0,\psi_1\}$ defined in Lemma I.5.18 of \cite{DemaillyBook}. By Proposition \ref{Convexity}, the $\Upsilon$-cone on the space of Hermitian matrices is convex. It also passes the PRT by the Courant–Fischer-Weyl min–max principle. For example, see Proposition 2.21 of \cite{ChenNieXuNumericalCriterion}.
            Then it is standard to see that $\underline{\varphi}$ is a smooth function satisfying the boundary conditions and $\pi_{X}^{*}\chi+\sqrt{-1}D\bar{D}\underline{\varphi} \in \Upsilon_{\tilde g_2}(\omega_\epsilon)$ for all $\epsilon<1$. For example, see Page 561-562 of \cite{ChenJEquationDHYM} for the computation of the complex Hessian of the regularized maximum function. Note that $\underline{\varphi}=\psi_0$ in a neighborhood of $\{\lvert z_0\rvert=1\}$ and $\underline{\varphi}=\psi_1$ in a neighborhood of $\{\lvert z_0\rvert=e\}$. This provides a bound on $|\frac{\partial}{\partial t}\underline{\varphi}|$ on $\partial\mathcal{X}$.

            Finally, since $\pi^*_X\chi + \sqrt{-1}D\bar{D}\underline{\varphi} \in \Upsilon_{\tilde g_2}(\omega_{\epsilon}) \subset \Upsilon_{\tilde g_{2,s}}(\omega_{\epsilon})$, the maximum principle implies that $\underline{\varphi} \le \varphi_{\epsilon, s}$.            
        \end{proof}
        
    As a corollary, we get the following bound on $|\frac{d}{dt}\varphi_{\epsilon, s}|$.

    \begin{prop}
            If $ \varphi_{\epsilon, s} : \mathcal{X} \to \mathbb{R} $ solves \eqref{ContinuityPath}, then there exists a constant $C$ depending only on the upper bound of $\lambda(\omega^{-1}\chi_{\varphi_i})$, $i=0,1$, and the positive lower bound of $g_2(\lambda(\omega^{-1}\chi_{\varphi_i}))$, $i=0,1$  such that
            \[
            |\frac{d}{dt}\varphi_{\epsilon, s}|\le C(\| \varphi_0 - \varphi_1 \|_{L^{\infty}(X)}+\epsilon^2)\]
            on $\mathcal{X}$.
    \label{ddtestimate}
    \end{prop}
    \begin{proof}
   By Lemma \ref{supsubsolutions}, using the facts that $ \underline{\varphi} \leq \varphi_{\epsilon, s} \leq \overline{\varphi}$ holds on $\mathcal{X}$ and that these functions agree on $\partial\mathcal{X}$, we see that
    \[
    |\frac{d}{dt}\varphi_{\epsilon, s}|\le \max\{|\frac{d}{dt}\overline{\varphi}|, |\frac{d}{dt}\underline{\varphi}|\}
    \le C(\| \varphi_0 - \varphi_1 \|_{L^{\infty}(X)}+\epsilon^2)
    \]
    on $\partial\mathcal{X}$.
    From the K\"ahler property, we have $\frac{d^{2}\varphi_{\epsilon, s}}{dt^{2}}\ge 0$. This implies that $\frac{d\varphi_{\epsilon, s}}{dt}$ is non-decreasing in $t$, and the maximum of its absolute value occurs at an endpoint that is on $\partial\mathcal{X}$. This finishes the proof.
    \end{proof}

\section{Complex Hessian Estimates}
\label{C11section}
        In this section, we prove $C^{1,\bar 1}$ estimates which depend on $\epsilon$ but are independent of $s$.  We will fix $\epsilon>0$ and use the supersolution and subsolution in Lemma \ref{supsubsolutions}. All constants in this section will depend on $\epsilon$, $X$, $\chi$, $\omega$, $\lvert\varphi_0\rvert_{C^{\infty}}$, $\lvert\varphi_1\rvert_{C^{\infty}}$, and the positive lower bound of $g_2(\lambda(\omega^{-1}\chi_{\varphi_i}))$, $i=0,1$.  To simplify notations, we also write the solution $\varphi_{\epsilon,s}$ in \eqref{ContinuityPath} as $\varphi_{s}$, and write $\chi_{\varphi_{\epsilon,s}}=\pi_X^*\chi+\sqrt{-1}D\bar D\varphi_{\epsilon, s}$ as $\chi_{s}$.
        We will also write this equation\eqref{ContinuityPath} as $1-h_s(\chi_s)=0$, where 
        \[h_s(\chi_s)=\frac{-\sum_{k=0}^{n} \tilde{c}_k(s) \binom{n+1}{k} (\pi^*_{X}\chi + \sqrt{-1}D\bar{D}\varphi_{s} )^{k}\wedge (\pi_{X}^*\omega + \epsilon^{2}\sqrt{-1}dz_0\wedge  d\bar z_0)^{n+1-k}}{(\pi^*_{X}\chi + \sqrt{-1}D\bar{D}\varphi_s)^{n+1}}.\]
        Observe that this function $h_s$ is the same as that in \cite{LinInverseSigmaSolvability}. Thus, we can use the properties established there.
        
        The linearisation of $1-h_s(\chi_\varphi)$ is given by 
            \[\mathcal{L}_s(\psi)=-Dh_s(\psi)=-h_s^{i\bar{j}}\psi_{i\bar{j}},\]
        where $Dh_s$ is the linearisation of $h_s$.

        For a point $p\in\partial\mathcal{X}= X \times \partial\mathcal{A},$
        without loss of generality we can assume that $p\in X \times \{ t=0 \}.$
        Near $p$, write $\mathcal{A}$ as $\{t\ge 0\}$.
        Let $B_{\delta}(0)\subset X$ be the $\delta$-ball in $X$ centered at $p$ and $z_{X} = (z_1,\cdots,z_n) $ be the coordinates on $X$.
        Set $\Omega_{\delta}= B_{\delta}(0) \times \{ 0\le t\le \delta \}$,
        where $t= \log(\lvert z_0\rvert)$.        
        Consider the function $\varphi_s-\underline{\varphi}$.
        Put
        \[
        T_s:=\sum_{i=0}^{n}(-h_s^{i\bar i}).
        \]

        For any $0\le s \le 1$, Lemma \ref{PathRightNoeth} implies that $\Upsilon_{\tilde{g}_{2}}^{i}\subset\Upsilon_{\tilde{g}_{2,s}}^{i}$. Also, we have the solution $\varphi_{s}$, and $\underline{\varphi} \in \Upsilon^{1}_{\tilde{g}_2}(\omega_{\epsilon})$.
        So Lemma~3.7 of \cite{LinInverseSigmaSolvability} gives uniform
        constants $N>0$ and $\kappa>0$ such that, on the set
        $\{\lambda_1>N\}$,
        \begin{equation}
        \mathcal{L}_s(\varphi_s-\underline{\varphi})=(-h_s^{i\bar{j}})(\varphi_s-\underline{\varphi})_{i\bar{j}}=-h_s^{i\bar{j}}(\varphi_{s,i\bar{j}}-\underline{\varphi}_{i\bar{j}})
        \le-\kappa T_s,
        \label{Lv-estimate} 
        \end{equation}
        where $\lambda_1\ge...\ge\lambda_{n+1}$ are the eigenvalues of $\omega_\epsilon^{-1}(\pi^*\chi+\sqrt{-1}D\bar D\varphi_s)$.
        \begin{remark}
            By Lemma~3.7 of \cite{LinInverseSigmaSolvability}, the constants $N,\kappa$ depend on the subsolution $\underline{\varphi}$, the coefficients of $\tilde{f}_2$ and the  manifold $\mathcal{X}$.
        \end{remark}
        
        We handle the complementary region $\{\lambda_1 \le N\}$ in the following way. By the assumption that $\tilde c_{n}(s)=0$ and Lemma \ref{rewrite-equation}, the region \[\{\lambda_1 \le N\}\cap\Upsilon^1_{\tilde g_{2,s}}\cap\{\tilde g_{2,s}=0\}=\partial\Upsilon_{\tilde g_{2,s}} \cap \{0 \le \lambda_{n+1}\le ...\le \lambda_1 \le N\}\] is a compact subset of $\Upsilon^1_{\tilde g_{2,s}}$. So $\partial_{\lambda_i} \tilde g_{2,s}$ has a positive lower bound and an upper bound. Consequently, by the compactness of $s\in [0,1]$, the linearisation is uniformly elliptic there. In particular, there are uniform constants $c_N,C_N>0$
        such that
        \begin{equation}
        -h_s^{0\overline{0}}\ge c_NT_s,
        \qquad
        |\mathcal{L}_s(\varphi_s-\underline{\varphi})|\le C_NT_s
        \qquad\text{on }\{\lambda_1\le N\}.
        \label{bounded-eigenvalue-ellipticity}
        \end{equation}

        \begin{lem}
        \label{mixed-derivative-bound}
            Let $p$ be any point on $\partial \mathcal{X}$.
            There exists a constant $C$ independent of $s$ such that 
            \[\left\lvert\frac{\partial^{2}\varphi_s}{\partial z_{i}\partial\bar{z}_{0}}(p)\right\rvert\le C  \left(\max_{X\times\mathcal{A}}\lvert\nabla\varphi_s\rvert+1\right),\]
            for all $i=1,2,...,n$.
        \end{lem}
        \begin{proof}
            Let $\tilde{D}$ be any constant linear first-order operator (in the spatial directions only) near the boundary, and set $M_s=\max_{X\times\mathcal{A}}\big(\lvert\nabla\varphi_s\rvert+1\big)$.  On the product neighborhood $\Omega_{\delta} =  B_{\delta}(0) \times \{ 0\le t\le \delta \}$, put
            \[
            \rho=t(2\delta-t).
            \]
            Thus $\rho\ge0$ on $\Omega_\delta$, $\rho=0$ at $p$, and
            \begin{equation}
            \mathcal{L}_s(\rho)= \frac{1}{2\lvert z_0\rvert^2} h_s^{0\overline{0}}\le \frac{1}{2e^{2}}h_s^{0\overline{0}}\le 0
            \label{rho-barrier}
            \end{equation}
            Choose positive constants $A = aM_s$, $B =bM_s$, and $K = kM_s$, each a sufficiently large multiple of $M_s$, in the order specified below, and consider
            \[
            w_s=A(\varphi_s-\underline{\varphi})+B\lvert z_X\rvert^{2}+K\rho
            -\tilde{D}(\varphi_s-\underline{\varphi}).
            \]
            On the $t=0$ part of the boundary,
            $\tilde D(\varphi_s-\underline\varphi)=0$.  
            On the $t=\delta$ part of the boundary, $\rho =\delta^2$. We choose $k>k_1$, $k_1$ sufficiently large, so that
            $K\rho = kM_s\delta^2 \ge |\tilde D(\varphi_s-\underline\varphi)|$.
            On the $\partial B_{\delta}(0)$ part of the boundary,
            we  choose $b$ sufficiently large, so that
            $B\lvert z_X\rvert^{2} = bM_s\delta^2\ge |\tilde D(\varphi_s-\underline\varphi)|$.
            Since $\varphi_s \ge \underline{\varphi}$, it follows that $w_s\ge0$ on
            $\partial\Omega_\delta$, while $w_s(0)=0$.
            Here we can use the compactness of $s$ to obtain the uniform lower bound of $k,b$.

           We next estimate $\mathcal L_s(w_s)$.  
            By computation, $L_s(B|z_X|^2)$ is bounded by a
            constant times $BT_s$.  We now establish the estimate
            for the differentiated term. It is easy to see that for any Hermitian matrices $A$, $B$, $C$ with $A$ positive definite, we have
            \[\frac{d}{dt}|_{t=0}\lambda((A+tC)^{-1}B)=\frac{d}{dt}|_{t=0}\lambda(A^{-1}(B+tD))\]
            for $D=-\frac{1}{2}(CA^{-1}B+BA^{-1}C)$.
            Consider the equation
            $1-h_s(\chi_s)=0$. Recall that it only depends on $\lambda(\omega_\epsilon^{-1}\chi_s)$, and $h_s^{ij}$ means the derivative of $h_s$ when we fix $\omega_\epsilon$ and differentiate $\chi_s$.
            
           So
           \[
                 0=-\tilde{D}(h_s(\chi_s))=-h_s^{i\bar{j}}(\tilde{D}\chi_{s,i\bar{j}}-\frac{1}{2}(\tilde{D}\omega_\epsilon)_{i\bar k}\omega_\epsilon^{\bar kl}\chi_{s, l\bar j}-\frac{1}{2}\chi_{s, i\bar k}\omega_\epsilon^{\bar kl}(\tilde{D}\omega_\epsilon)_{l\bar j})
           \]
           By (3.19) of \cite{LinInverseSigmaSolvability},
           \[ |h_s^{i\bar{j}}\tilde{D}(\pi_{X}^{*}\chi)_{i\bar{j}}+h_s^{i\bar{j}}\tilde{D}\varphi_{s, i\bar{j}}|=|h_s^{i\bar{j}}\tilde{D}\chi_{s,i\bar{j}}|=\frac{1}{2}|h_s^{i\bar{j}}((\tilde{D}\omega_\epsilon)_{i\bar k}\omega_\epsilon^{\bar kl}\chi_{s, l\bar j}+\chi_{s, i\bar k}\omega_\epsilon^{\bar kl}(\tilde{D}\omega_\epsilon)_{l\bar j})| \le C\]
           for a constant $C$ independent of $s$.
            We now have
            \begin{equation}
                -\mathcal{L}_s(\tilde{D}(\varphi_s-\underline{\varphi}))=h_s^{i\bar{j}}\tilde{D}\varphi_{s, i\bar{j}}-h_s^{i\bar{j}}\tilde{D}(\underline{\varphi})_{i\bar{j}}
                \le CT_s+C.
            \end{equation}
            By (3.13) of \cite{LinInverseSigmaSolvability}, there exists a constant $C$ independent of $s$ such that $T_s\ge C>0$. For a different constant $C$, this implies that
            \begin{equation}
            \label{LD}
                -\mathcal{L}_s(\tilde{D}(\varphi_s-\underline{\varphi})) \le CT_s
            \end{equation}
            On $\{\lambda_1>N\}$, equations \eqref{Lv-estimate},
            \eqref{rho-barrier} and \eqref{LD} give
            \[
            \mathcal L_s(w_s)\le(-A\kappa+C(B+1))T_s.
            \]
            After $B$ has been fixed, choose $A=aM_s$ with $a$ sufficiently
            large so that the right-hand side is negative.  On
            $\{\lambda_1\le N\}$, equations
            \eqref{bounded-eigenvalue-ellipticity}, \eqref{rho-barrier}
            and \eqref{LD}
            give
            \[
            \mathcal{L}_s(w_s)
            \le \big(C_NA+C_1(B+1)-\tfrac{1}{2e^2}c_NK\big)T_s.
            \]
            Finally choose $K=kM_s$ with $k > k_2$, $k_2$ sufficiently large, so that
            this expression is negative as well.  Therefore
            $\mathcal{L}_s(w_s)<0$ throughout $\Omega_\delta$.
            Actually we need to choose $K=kM_s$ with $k>\max(k_1,k_2).$

            The maximum principle implies that $w_s\ge 0$ in $\Omega_{\delta}$ . Since $w_s(0)=0$, the Hopf boundary lemma gives $\frac{\partial w_s}{\partial t}\ge 0$. Hence, using Proposition \ref{ddtestimate},
            \[\frac{\partial}{\partial t}\tilde{D}\varphi_s(0)\le C_2 M_s\]
            for some constant $C_2$. Since $\tilde{D}$ is any first-order operator near $\partial(X\times \mathcal{A})$, replacing $\tilde{D}$ with $-\tilde{D}$ gives
            \[-\frac{\partial}{\partial t}\tilde{D}\varphi_s(0)\le C_2 M_s.\]
            Using the $S^1$-invariance, \[\left\lvert \frac{\partial}{\partial z_{0}}\tilde{D}\varphi_s(0)\right\rvert <C_3 M_s.\]

            We conclude that
            \[\left\lvert\frac{\partial^{2}\varphi_s}{\partial z_{i}\partial\bar{z}_{0}}(p)\right\rvert\le C_4 \left(\max_{X\times\mathcal{A}}\lvert\nabla\varphi_s\rvert+1\right).\]
        \end{proof}

        Now we prove the $C^{1, \bar 1}$-estimate compared to the $C^1$-norm.
        
        \begin{prop}
        There exists a constant $C$ such that 
        \[
        \max_{X\times\mathcal{A}}\lvert D\bar D\varphi_s\rvert \le C \left(\max_{X\times\mathcal{A}}\lvert\nabla\varphi_s\rvert^{2}+1\right).
        \]
        \end{prop}
        \begin{proof}
        We first show that 
        \[
        \max_{\partial(X\times\mathcal{A})}\lvert D\bar D\varphi_s\rvert \le C \left(\max_{X\times\mathcal{A}}\lvert\nabla\varphi_s\rvert^{2}+1\right).
        \]
        The spatial direction estimates come from the boundary condition that $\varphi_s$ restricts to the given potentials $\varphi_0$ and $\varphi_1$. The mixed direction estimates come from Lemma \ref{mixed-derivative-bound}. To get an estimate of $\frac{d^{2}\varphi_s}{dt^{2}}$ on the boundary, we extend Lemma \ref{geodesic lemma} from $s=1$ to $s\in[0,1]$ to get the expression of $\frac{d^{2}\varphi_s}{dt^{2}}$ and then use the estimates in the spatial and mixed directions. By compactness, the estimate is uniform in $s$.

        For the interior estimate, we choose a large enough $A$ and consider the function
        \[
        F_s = -A(\varphi_s-\underline{\varphi}) + \log (1+\lambda_1),
        \]        
        where $\lambda_1\ge...\ge\lambda_{n+1}$ are the eigenvalues of $\omega_\epsilon^{-1}(\pi^*\chi+\sqrt{-1}D\bar D\varphi_s)$.
        If $\max_{\overline{\mathcal{X}}}F_s$ is achieved on $\partial\mathcal{X}$, then by the $C^0$-estimate and the boundary estimate, we get the required bound. If $\max_{\overline{\mathcal{X}}}F_s$ is achieved in the interior, then we must use the perturbation technique in \cite{LinInverseSigmaSolvability} since the eigenvalue is in general not differentiable. Then we follow Theorem 3.1 of \cite{LinInverseSigmaSolvability} to get the required estimate.        
        \end{proof}

        Now we use the blow-up argument to prove the gradient estimate and the complex Hessian estimate:

        \begin{prop}
        There exists a constant $C$ independent of $s$ such that 
        \[
        \max_{X\times\mathcal{A}}\lvert D\bar D\varphi_s\rvert \le C, \quad \max_{X\times\mathcal{A}}\lvert\nabla\varphi_s\rvert\le C.
        \]
        \end{prop}

        \begin{proof}
       Fix $\epsilon$. Assume, to the contrary, that there are $s_i\in[0,1]$ and points $x_{i}\in X\times\overline{\mathcal{A}}$ such that $c_{i}:=\sup\lvert\nabla\varphi_{s_i}\rvert=\lvert\nabla\varphi_{s_i}(x_i)\rvert\rightarrow\infty$. Since $X\times \bar{\mathcal{A}}$ and $[0,1]$ are compact, after passing to a subsequence, we may assume that $x_{i}\rightarrow x_{\infty}$ and $s_i\rightarrow s_\infty$. Let $U_i$ be a ball of radius $\delta>0$ with respect to $\omega_{\epsilon}$ around $x_i$, and let $\widehat U_i=c_i(U_i-x_i)$. We define $\hat{\varphi}_{i}(z)=\varphi_{s_i}(x_{i}+\frac{z}{c_{i}})$. Then $\lvert\nabla\hat{\varphi}_{i}\rvert\le 1$ for all $z\in\widehat U_{i}$ and $\lvert\nabla\hat{\varphi}_{i}(0)\rvert=1$. Then we have $\lvert\Delta \hat{\varphi}_{i}\rvert\le C$. We have $\underline{\varphi}\le \varphi_{s_i}\le \overline{\varphi}$ for all $i$. Rescale $\underline{\varphi}$ and $\overline{\varphi}$ accordingly by setting $\hat{\underline{\varphi}}(z)=\underline{\varphi}(x_{i}+\frac{z}{c_{i}})$ and $\hat{\overline{\varphi}}(z)=\overline{\varphi}(x_{i}+\frac{z}{c_{i}})$. Then $\hat{\underline{\varphi}}\le \hat{\varphi}_{i}\le \hat{\overline{\varphi}}$ for all $i$. By the standard elliptic estimate, we have $\hat{\varphi}_{i}\rightarrow \varphi_{\infty}$ in $C^{1,\eta_0}$ for some $\eta_0>0$ on every compact subset of the limit space $\mathbb{C}^{n+1}$ or half-space $\mathbb{C}^{n+1}\cap\{\Re z_0 \ge 0\}$. This implies
        \[\lvert\nabla\varphi_{\infty}(0)\rvert=1\]
        and 
        \[\underline{\varphi}(x_\infty)\le \varphi_{\infty}(x)\le \overline{\varphi}(x_\infty)\]
        for all $x$. 
        
        If $x_\infty\in \partial(X\times \mathcal{A})$, then $\underline{\varphi}(x_\infty)=\overline{\varphi}(x_\infty)$ because they both equal the boundary function $\varphi_0$ or $\varphi_1$. Thus, $\varphi_{\infty}$ is constant, a contradiction.
        
        If $x_\infty\not\in\partial(X\times \mathcal{A})$, then we can apply the arguments in Section 5 of \cite{Collins201511FW}. We have the K\"ahler form
        \[\chi_{s_i}=\pi_{X}^{*}\chi+\sqrt{-1}D\bar{D}\varphi_{s_i}>0.\] After rescaling and passing to the limit, this gives $\sqrt{-1}D\bar{D}\varphi_{\infty}\ge 0$ in the sense of distribution. Since $\varphi_\infty$ is a bounded plurisubharmonic function on $\mathbb{C}^{n+1}$, it is a constant, again a contradiction.
        \end{proof}

\section{Existence of the Weak Geodesic}
 
 We still fix $\epsilon$. By Section 6 of \cite{Collins201511FW}, Lin's convexity of the graph of complex Hessians \cite{LinInverseSigmaConvexity} can be rewritten as a concave and uniformly elliptic operator on real Hessians \cite{Collins201511FW}. Then the Evans-Krylov theory in \cite{EvansClassicalFullyNonlinear} and \cite{KrylovBoundaryHolder} implies the $C^{2,\alpha}$ bound. Then the standard Schauder estimate implies a $C^{k,\alpha}$ bound for all $k$. Now we can prove the closedness of $P$ defined in Section \ref{SecContinuityPath} and obtain that $P = [0,1]$.
Then the smooth $\epsilon$-geodesic always exists.

 The uniqueness is standard because the space $\Upsilon^1_{\tilde g}$ is convex and the $\epsilon$-geodesic equation is an elliptic equation on $\Upsilon^1_{\tilde g}$. For example, see Corollary 17.2 in \cite{Giltru} for the proof. This also implies that the solution is $S^1$-invariant.

Now we prove the monotonicity of $\epsilon$-geodesics when $\epsilon\to 0$.
Recall that 
\[
\omega_{\epsilon} = \pi_{X}^*\omega + \epsilon^2\sqrt{-1} dz_0 \wedge d\bar z_0,
\]
and the $\epsilon$-geodesic $\chi_\epsilon=\pi_{X}^{*}\chi+\sqrt{-1}D\bar{D}\varphi_\epsilon$ is K\"ahler. By the Courant-Fischer-Weyl min-max principle
\[
  \lambda_i(A^{-1}B)
  =\max_{\substack{V\subseteq\mathbb{C}^{n+1} \\ \dim V=i}}
    \min_{0\neq x\in V}\frac{x^*Bx}{x^*Ax},
\]
and the condition that the $\Upsilon$-cone pass the PRT, we see that for any $0<\epsilon_1\le\epsilon_2$, the $\epsilon$-geodesic equation $\lambda(\omega_{\epsilon_2}^{-1}\chi_{\epsilon_2})\in \partial\Upsilon_{\tilde{g}}\cap\Upsilon^1_{\tilde{g}}$ implies that $\lambda(\omega_{\epsilon_1}^{-1}\chi_{\epsilon_2})\in \overline\Upsilon_{\tilde{g}}$. The maximum principle and the boundary condition then imply the comparison $\varphi_{\epsilon_1}\ge \varphi_{\epsilon_2}$. This finishes the proof of Theorem \ref{MainThm1}.

The monotonicity of the $\epsilon$-geodesic and the uniform $C^0$ estimate implied by the boundary condition and Proposition \ref{ddtestimate} give an $S^1$-invariant function
\[
\varphi
=\left(\lim_{\epsilon\to 0}\varphi_{\epsilon}\right)^*
=\left(\sup_{\epsilon>0}\varphi_\epsilon\right)^* \in \operatorname{PSH}(X\times\mathcal A,\pi_X^*\chi)
\cap L^\infty(X\times\mathcal A),
\]
where $^*$ denotes upper-semicontinuous regularization, and we write
\[
\varphi_t=\varphi|_{X\times\{|z_0|=e^t\}}\in \operatorname{PSH}(X,\chi)\cap L^\infty(X)
\]
for all $t\in[0,1]$. 

By taking the limit of the uniform estimates in Proposition \ref{ddtestimate}, we see that
\[
\|\varphi_t-\varphi_s\|_{L^\infty(X)}\le C|t-s|.
\]
As a corollary, we provide a different characterization of $\varphi_t$. We can also first restrict $\varphi_\epsilon$ to a slice of $X\times \mathcal{A}$ to obtain $\varphi_{\epsilon,t}(x)=\varphi_{\epsilon}(x,t)$, then take the upper-semicontinuous regularization of its pointwise limit \[\hat\varphi_{t}=\left(\lim_{\epsilon\to 0}\varphi_{\epsilon, t} \right)^*\in \operatorname{PSH}(X,\chi)\cap L^\infty(X).\] Fix $(x_{0},t_{0})\in X\times \mathcal{A}$. Using Proposition \ref{ddtestimate} we have $$\varphi_\epsilon(x,t)\le \varphi_\epsilon(x,t_{0})+C|t-t_{0}|.$$ Therefore \[
\varphi_{t_{0}}(x_{0}) =\varphi(x_{0},t_{0}) =\limsup_{(x,t)\to(x_0,t_0)}\lim_{\epsilon \rightarrow 0}\varphi_\epsilon(x,t) \le \limsup_{x\rightarrow x_{0}}\lim_{\epsilon \rightarrow 0} \varphi_\epsilon(x,t_{0}) = \hat \varphi_{t_{0}}(x_{0}).
\] 
For the reverse inequality $$\hat \varphi_{t_{0}}(x_{0})=\limsup_{x\rightarrow x_{0}}\lim_{\epsilon \rightarrow 0}\varphi_\epsilon(x,t_{0}) \le \limsup_{(x,t)\rightarrow (x_{0},t_{0})}\lim_{\epsilon \rightarrow 0}\varphi_\epsilon(x,t)=\varphi(x_{0},t_{0}).$$ This proves that  $\varphi_{t}=\hat\varphi_{t}$. 

All the functions $\varphi_\epsilon$ and $\varphi_{\epsilon,t}$ are uniformly bounded, monotone $\pi_{X}^{*}\chi$-psh or $\chi$-psh functions. So we can use the Bedford-Taylor theory \cite{BedfordTaylorCapacity} to study them. We rewrite the $\epsilon$-geodesic equation as \[
\sum_{k=1}^{n+1}\tilde{c}_{k}\binom{n+1}{k}(\pi_{X}^{*}\chi+\sqrt{-1}D\bar{D}\varphi_\epsilon)^{k}\wedge(\pi_{X}^{*}\omega)^{n+1-k}=-(n+1)\epsilon^{2}\sqrt{-1}dz_0\wedge d\bar z_0\wedge\tilde g_2(\chi_{\varphi_\epsilon}),
\]  
where \[\tilde g_2(\chi_{\varphi_\epsilon})=\sum_{k=0}^{n}\tilde{c}_{k}\binom{n}{k}\big(\pi_{X}^{*}\chi+\sqrt{-1}\partial\bar{\partial}\varphi_{\epsilon}\big)^{k}\wedge(\pi_{X}^{*}\omega)^{n-k}.\]

Observe that the term $-\epsilon^{2}\sqrt{-1}dz_{0}\wedge d\bar{z_{0}}\wedge \tilde g_2(\chi_{\varphi_\epsilon})$ is non-negative by Proposition \ref{extension}. Consider any continuous test function $\eta$ on $\overline{\mathcal{X}}$. Then we have $$|\int_{X\times \bar{\mathcal{A}}}\eta(-\epsilon^{2}\sqrt{-1}dz_{0}\wedge d\bar{z_{0}}\wedge \tilde g_2(\chi_{\varphi_\epsilon}))|\le C\epsilon^{2}\|\eta\|_{C^0(X\times\bar{\mathcal{A}})},$$ where $C$ is the cohomological constant $-\int_{X\times \bar{\mathcal{A}}}\sqrt{-1}dz_{0}\wedge d\bar{z_{0}}\wedge \tilde g_2(\chi_{\varphi_\epsilon})$. Now taking $\epsilon\rightarrow 0$ we get that $-\epsilon^{2}\sqrt{-1}dz_{0}\wedge d\bar{z_{0}}\wedge \tilde g_2(\chi_{\varphi_\epsilon})$ tends weakly to zero as measures. This gives $$\sum_{k=1}^{n+1}\tilde{c}_{k}\binom{n+1}{k}(\pi_{X}^{*}\chi+\sqrt{-1}D\bar{D}\varphi)^{k}\wedge(\pi_{X}^{*}\omega)^{n+1-k}=0.$$
So $\varphi$ is a weak geodesic satisfying the geodesic equation in the Bedford-Taylor sense.

  \section{The Functional and the Convexity Condition}
        
       For $i=1, 2$, we define the functional $I_i: C^{\infty}(X,\mathbb{R})\rightarrow \mathbb{R}$ by
        \begin{equation}
        \label{def-I}
            I_i(\varphi) = \int_{X}\varphi\sum_{k=0}^{n}\binom{n}{k}c_k^{(i)}\sum_{j=0}^{k}\frac{1}{j+1}\binom{k}{j}\chi^{k-j}\wedge (\sqrt{-1}\partial\bar{\partial}\varphi)^{j}\wedge\omega^{n-k},
        \end{equation}
        and define $J_i=-I_i$.

        Let $\varphi(t)$ be any smooth path. We compute the derivative and see that
         \begin{equation}
         \begin{split}
            \frac{d}{dt}I_i(\varphi) &= \int_{X}\frac{d\varphi}{dt}\sum_{k=0}^{n}\binom{n}{k}c_k^{(i)}\sum_{j=0}^{k}\frac{1}{j+1}\binom{k}{j}\chi^{k-j}\wedge (\sqrt{-1}\partial\bar{\partial}\varphi)^{j}\wedge\omega^{n-k}\\
            &+\int_{X}\varphi\sum_{k=0}^{n}\binom{n}{k}c_k^{(i)}\sum_{j=1}^{k}\frac{j}{j+1}\binom{k}{j}\chi^{k-j}\wedge (\sqrt{-1}\partial\bar{\partial}\varphi)^{j-1}\wedge\omega^{n-k}\wedge\sqrt{-1}\partial\bar{\partial}\frac{d\varphi}{dt}\\
            &=\int_{X}\frac{d\varphi}{dt}\sum_{k=0}^{n}\binom{n}{k}c_k^{(i)}\sum_{j=0}^{k}(\frac{1}{j+1}+\frac{j}{j+1})\binom{k}{j}\chi^{k-j}\wedge (\sqrt{-1}\partial\bar{\partial}\varphi)^{j}\wedge\omega^{n-k}\\
            &=\int_{X}\frac{d\varphi}{dt}\sum_{k=0}^{n}\binom{n}{k}c_k^{(i)}\chi_\varphi^{k}\wedge\omega^{n-k}\\
            &=\int_{X}\frac{d\varphi}{dt}g_i(\chi_\varphi).
        \end{split}
        \label{ddtI}
        \end{equation}

        See \cite{ChenTianRicciFlow} and \cite{FangLaiMaFlows} for more discussions about these functionals. In Xiuxiong Chen's case \cite{ChenMabuchiEnergy}, $f_1(x)=x^n-cx^{n-1}$ and $f_2(x)=x^n$. Then the $J_1$ functional is the $J$-functional, and $I_2$ functional is the $I$-functional. In Collins-Yau's case \cite{CollinsYauGeodesics},
        \[
        f_1(x)=\Im(e^{-\sqrt{-1}\hat\theta}(1+\sqrt{-1}x)^n), \quad f_2(x)=\Re(e^{-\sqrt{-1}\hat\theta}(1+\sqrt{-1}x)^n).
        \]
        Then the $J_1$ functional is the $\mathcal{J}$-functional, and the $I_2$ functional is the $\mathcal{C}$-functional in their paper.
        Now suppose that $\varphi_0$ and $\varphi_1$ are smooth potentials in $\mathcal{H}$, and $\varphi_{\epsilon,t}$ is the smooth $\epsilon$-geodesic connecting them. Taking one more derivative of \eqref{ddtI}, we see that
        \[
            \frac{d^2I_i(\varphi_{\epsilon,t})}{dt^2}
            =\int_X\frac{d^2\varphi_{\epsilon,t}}{dt^2}g_i(\chi_{\varphi_{\epsilon,t}})
            +\int_X\frac{d\varphi_{\epsilon,t}}{dt}\sqrt{-1}\partial\bar{\partial}\frac{d\varphi_{\epsilon,t}}{dt}\wedge g_i'(\chi_{\varphi_{\epsilon,t}}),
        \]
        where
        \[
            g_i'(\chi_{\varphi_{\epsilon,t}})=\sum_{k=1}^n kc_k^{(i)}\binom{n}{k}\chi_{\varphi_{\epsilon,t}}^{k-1}\wedge\omega^{n-k}.
        \]
        Integrating by parts, we obtain
        \[
            \frac{d^2I_i(\varphi_{\epsilon,t})}{dt^2}
            =\int_X\frac{d^2\varphi_{\epsilon,t}}{dt^2}g_i(\chi_{\varphi_{\epsilon,t}})
            -\int_X\sqrt{-1}\partial\frac{d\varphi_{\epsilon,t}}{dt}\wedge\bar{\partial}\frac{d\varphi_{\epsilon,t}}{dt}\wedge g_i'(\chi_{\varphi_{\epsilon,t}}).
        \]

        From the $\epsilon$-geodesic equation in Lemma \ref{geodesic lemma}, we have
        \[
            \frac{d^2I_2(\varphi_{\epsilon,t})}{dt^2}=\int_X -4\epsilon^2e^{2t}\left(\sum_{k=0}^{n}\tilde{c}_{k}\binom{n}{k}\big(\pi_{X}^{*}\chi+\sqrt{-1}\partial\bar{\partial}\varphi_{\epsilon,t}\big)^{k}\wedge(\pi_{X}^{*}\omega)^{n-k}\right)=C\epsilon^2e^{2t},
        \]
        and
        \begin{equation}
        \begin{split}
            -\frac{d^2J_1(\varphi_{\epsilon,t})}{dt^2}
            &=\frac{d^2I_1(\varphi_{\epsilon,t})}{dt^2}\\
            &=\int_X\sqrt{-1}\partial\frac{d\varphi_{\epsilon,t}}{dt}\wedge\bar{\partial}\frac{d\varphi_{\epsilon,t}}{dt}\wedge
            \left(\frac{g_1(\chi_{\varphi_{\epsilon,t}})}{g_2(\chi_{\varphi_{\epsilon,t}})}g_2'(\chi_{\varphi_{\epsilon,t}})-g_1'(\chi_{\varphi_{\epsilon,t}})\right)\\
            &+\int_X -4\epsilon^2e^{2t}\left(\sum_{k=0}^{n}\tilde{c}_{k}\binom{n}{k}\big(\pi_{X}^{*}\chi+\sqrt{-1}\partial\bar{\partial}\varphi_{\epsilon,t}\big)^{k}\wedge(\pi_{X}^{*}\omega)^{n-k}\right)\frac{g_1(\chi_{\varphi_{\epsilon,t}})}{g_2(\chi_{\varphi_{\epsilon,t}})}.
            \label{ConvexEq3}
        \end{split}
        \end{equation}

        Here the quotient of two top-degree forms denotes the quotient of their scalar densities with respect to any fixed positive local volume form. 
        
            We first view
            \[
            \frac{g_1(\chi_\varphi)}{g_2(\chi_\varphi)}g_2'(\chi_\varphi) - g_1'(\chi_\varphi) \]
            as an $(n-1,n-1)$-form and determine its sign.
           \begin{align*}
                g_{i}'(\chi_\varphi)=&\sum_{k=1}^{n}kc_k^{(i)}\binom{n}{k}\chi_{\varphi}^{k-1}\wedge \omega^{n-k} \\
                =& \sum_{j=1}^{n} \sum_{k=1}^{n}kc_k^{(i)}\binom{n}{k}(k-1)!(n-k)! \sigma_{k-1}(\lambda_{;j}) \prod_{l\not=j}\sqrt{-1}{dz^{l}\wedge d\bar{z}^{l}} \\
                =& n!\sum_{j=1}^{n} \sum_{k=1}^{n}c_k^{(i)}\sigma_{k-1}(\lambda_{;j}) \prod_{l\not=j}\sqrt{-1}{dz^{l}\wedge d\bar{z}^{l}}
            \end{align*}

            Thus, the nonpositivity of this form is equivalent to the $n$ inequalities
            \begin{equation}
            g_1(\lambda)\partial_{\lambda_{j}}g_2(\lambda) -g_2(\lambda)\partial_{\lambda_{j}}g_1(\lambda) \le 0,\qquad 1\leq j\leq n.
            \label{derivative-monotonicity}
            \end{equation}

            This is exactly the third condition in Theorem~\ref{PositiveDerivative} because of $f_2\prec  f_1$.

        For the other term, since 
        $\partial_{\lambda_{j}} \frac{g_1(\lambda)}{g_2(\lambda)}\ge 0$ for all $\lambda \in \Upsilon_{g_2}$ by \eqref{derivative-monotonicity}, we see that $\frac{g_1(\lambda)}{g_2(\lambda)}$ is bounded above by its limit when all $\lambda_i\to \infty$. The limit is $1$ because $f_1$ and $f_2$ are monic polynomials with the same degree, so $\frac{g_1(\lambda)}{g_2(\lambda)}\le 1$ on $\Upsilon_{g_2}$.

        We substitute these estimates in \eqref{ConvexEq3} and see that
        \[
        -\frac{d^2J_1(\varphi_{\epsilon,t})}{dt^2}
        =\frac{d^2I_1(\varphi_{\epsilon,t})}{dt^2}\le C\epsilon^2e^{2t}.
        \]

        Now we start to study the weak geodesic $\varphi_t$. We can use the Bedford-Taylor products \cite{BedfordTaylorCapacity} in \eqref{def-I} to define $I_i(\varphi_t)$ and $J_i(\varphi_t)=-I_i(\varphi_t)$.
        Theorem 2.6 of \cite{BedfordTaylorCapacity} implies
\[
I_i(\varphi_{\epsilon,t})\longrightarrow I_i(\varphi_t),
\qquad
J_i(\varphi_{\epsilon,t})\longrightarrow J_i(\varphi_t), \qquad \epsilon\to 0.
\]
Now we claim that by \eqref{ddtI},
\begin{equation}
\label{Lipschitz}
|I_i(u)-I_i(v)|\le C\|u-v\|_{L^\infty(X)}
\end{equation}
for all $u, v\in \mathcal{H}$. To see this, suppose $u,v\in \mathcal{H}$ and we know that $\mathcal{H}$ is convex. Consider  the  path  $u_{t}=(1-t)v+tu$. We see that $\frac{du_{t}}{dt}=u-v$. So 
\[
    |I_i(u)-I_i(v)|\le \int_{0}^{1}|\frac{d}{dt}I_{i}(u_{t})|dt\le \int_{0}^{1}\int_{X}|\frac{du_{t}}{dt}||g_{i}(\chi_{u_t})|dt \le C||u-v||_{L^{\infty}(X)}.
\]
Here $|g_{i}(\chi_{u_t})|$ means the coefficients are $|c_{k}^{(i)}|$.
By taking the limit of \eqref{Lipschitz} and Proposition \ref{ddtestimate}, we see that
$I_i(\varphi_t)$ and $J_i(\varphi_t)$ are Lipschitz on $[0,1]$. We can also take the limit to finish the proof of Theorem \ref{MainThm3}. To see this, we recall $$\frac{d^{2}I_{2}(\varphi_{\epsilon,t})}{dt^{2}}=C\epsilon^{2}e^{2t}.$$ Now integrating twice, we obtain $$I_{2}(\varphi_{\epsilon,t})=C\epsilon^{2}\frac{e^{2t}}{4}+A_\epsilon t+B_\epsilon.$$ Taking $\epsilon\rightarrow 0$, we obtain $I_{2}(\varphi_{t})=At+B$. This gives $$I_{2}(\varphi_{t})=(1-t)I_{2}(\varphi_{0})+tI_{2}(\varphi_{1})$$ for all $t\in [0,1]$. Now let us take $0\le t_{1}<t_{2}<t_{3}\le 1$ and define $$L(t)=
\begin{cases}
\dfrac{t-t_1}{t_2-t_1}, & t_1\leq t\leq t_2, \\[6pt]
\dfrac{t_3-t}{t_3-t_2}, & t_2\leq t\leq t_3.
\end{cases}$$ We see that $L(t_{1})=L(t_{3})=0$ and $L(t_{2})=1$. First let us consider the function $$K:= \frac{J_{1}(\varphi_{\epsilon,t_{2}})-J_{1}(\varphi_{\epsilon,t_{1}})}{t_{2}-t_{1}}-\frac{J_{1}(\varphi_{\epsilon,t_{3}})-J_{1}(\varphi_{\epsilon,t_{2}})}{t_{3}-t_{2}}.$$ By the fundamental theorem of calculus, $$K=\int_{t_{1}}^{t_{3}}L'(t)\frac{dJ_{1}(\varphi_{\epsilon,t})}{dt}dt.$$ Now integrating by parts we obtain $$K=[L(t)\frac{dJ_{1}(\varphi_{\epsilon,t})}{dt}]|_{t_{1}}^{t_{3}}-\int_{t_{1}}^{t_{3}}L(t)\frac{d^{2}J_{1}(\varphi_{\epsilon,t})}{dt^{2}}dt=-\int_{t_{1}}^{t_{3}}L(t)\frac{d^{2}J_{1}(\varphi_{\epsilon,t})}{dt^{2}}dt$$ since $L(t_{1})=L(t_{3})=0$. Because $L(t)\ge 0$ on $[t_{1},t_{3}]$ and $\frac{d^{2}J_{1}(\varphi_{\epsilon,t})}{dt^{2}}\ge -C\epsilon^{2}e^{2t},$ we obtain $$K\le C\epsilon^{2}\int_{t_{1}}^{t_{3}}L(t)e^{2t}dt.$$ Now integrating and taking $\epsilon\rightarrow 0$, we obtain $$\frac{J_{1}(\varphi_{t_{2}})-J_{1}(\varphi_{t_{1}})}{t_{2}-t_{1}}-\frac{J_{1}(\varphi_{t_{3}})-J_{1}(\varphi_{t_{2}})}{t_{3}-t_{2}}\le 0.$$ Rearranging we obtain $$J_{1}(\varphi_{t_{2}})\le \frac{t_{3}-t_{2}}{t_{3}-t_{1}}J_{1}(\varphi_{t_{1}})+\frac{t_{2}-t_{1}}{t_{3}-t_{1}}J_{1}(\varphi_{t_{3}}).$$
In other words, the $I_2$-functional is affine and the $J_1$-functional is convex on weak geodesics.

\section{\texorpdfstring{$\mathcal{E}^{p}$}{Ep} Theory}
    	 
	            Recall that $c_n^{(i)}=1$, and $c_{n-1}^{(2)}=0$. We define
	   \[
	   f_i(x) = \sum_{k=0}^{n} c_k^{(i)} \binom{n}{k} x^k,
	   \quad g_i(\lambda) = \sum_{k=0}^{n} c_k^{(i)} \sigma_k(\lambda),
	   \]
    and
	   \[
		   g_i(\chi_\varphi) = \sum_{k=0}^{n} c_k^{(i)}\binom{n}{k} \chi_\varphi^k \wedge \omega^{n-k},\quad i=1,2.
		   \]
	   We have defined the $\epsilon$-geodesic equation
\[
\sum_{k=0}^{n+1} \tilde{c}_k \binom{n+1}{k} (\pi^*_{X}\chi + \sqrt{-1}D\bar{D}\varphi_\epsilon )^{k}\wedge (\pi_{X}^*\omega + \epsilon^{2}\sqrt{-1}dz_0\wedge  d\bar z_0)^{n+1-k} =0
\]
on $\mathcal{X}:=X\times \mathcal{A}$ where $\tilde c_k=c_{k-1}^{(2)}$ for $1\le k\le n+1$, and $\tilde c_0$ is chosen by Lemma \ref{IntegralUpsilonStable}.
We define \[\tilde g_2(\chi_{\varphi_\epsilon})=\sum_{k=0}^{n}\tilde{c}_{k}\binom{n}{k}\big(\pi_{X}^{*}\chi+\sqrt{-1}\partial\bar{\partial}\varphi_{\epsilon}\big)^{k}\wedge(\pi_{X}^{*}\omega)^{n-k}.\]
Then, by Proposition \ref{extension}, $\tilde g_2(\chi_\varphi)\le 0$.
	   
Define \[
\mathcal{H}=\{\varphi\in C^{\infty}(X,\mathbb{R}): \chi_\varphi=\chi+\sqrt{-1}\partial\bar\partial\varphi\in \Upsilon_{g_2}(\omega)\}.
\] Let $\varphi(t)$, $t\in [0,1]$, be a piecewise smooth path in $\mathcal{H}$. For $p\in[1,\infty)$ and $\xi\in T_\varphi\mathcal H$, define
\[ \|\xi\|_{p,\varphi}^p = \int_X |\xi|^p g_2(\chi_\varphi).\]
For a path $\varphi(t)$, set
\[E_p(\varphi,t)=\int_X\left|\frac{d\varphi}{dt}\right|^p g_2(\chi_\varphi),\qquad
\operatorname{length}_{p}(\varphi)=\int_{0}^{1}E_{p}(\varphi,t)^{\frac{1}{p}}dt.
\]
When $p=2$, this norm is induced by an inner product, and its geodesics satisfy the geodesic equation above.

The $d_p$ distance is defined by
\[
d_p(\varphi_0,\varphi_1) = \inf_{\substack{\varphi_t \in \mathcal{H}\\ \varphi_t|_{t=0}=\varphi_0,\,\varphi_t|_{t=1}=\varphi_1}}
\int_0^1 \left\|\frac{d\varphi_t}{dt}\right\|_{p,\varphi_t}dt.
\]
 To regularize the absolute-value power near zero, for  $0<\delta\le 1$, we introduce
 \[E_{p,\delta}(\varphi,t)=\int_{X}\sqrt{\left|\frac{d\varphi_{t}}{dt}\right|^{2p}+\delta^{2}}\, g_{2}({\chi_\varphi}), \quad \operatorname{length}_{p,\delta}(\varphi)=\int_{0}^{1}E_{p,\delta}(\varphi,t)^{\frac{1}{p}}dt.\]
 Then for any smooth path $\varphi(t)$ in $\mathcal{H}$, there exists $C(p,\chi,\omega,X)$ such that 
 \begin{equation}
 \label{approximate inequality}\lvert E_{p}(\varphi,t)-E_{p,\delta}(\varphi,t)\rvert\le C\delta, \quad \lvert \operatorname{length}_{p}(\varphi)-\operatorname{length}_{p,\delta}(\varphi)\rvert\le C\delta^{\frac{1}{p}}.
\end{equation}
 \begin{lem}
 \label{estimate lemma}
 	For any $\varphi_{0},\varphi_{1}\in \mathcal{H}$, let $\varphi_{\epsilon}$ be the $\epsilon$-geodesic connecting $\varphi_{0}$ and $\varphi_{1}$. Denote the corresponding energies by $E_{p,\delta}^{\epsilon}$. Then there exists $C(\varphi_{0},\varphi_{1},p,\chi,\omega,X)$ such that, for any $0<\delta\le 1$,\\
 	$(i)$ $ \lvert\partial_{t}E_{p,\delta}^{\epsilon}\rvert\le C\epsilon^{2}$.\\
 	$(ii)$ The following inequality holds 
 	\[E_{p,\delta}^{\epsilon}(t)\ge \max\left\{{\int_{\varphi_{0}>\varphi_{1}}\lvert \varphi_{0}-\varphi_{1}\rvert^{p}g_{2}(\chi_{\varphi_{0}})},{\int_{\varphi_{1}>\varphi_{0}}\lvert \varphi_{0}-\varphi_{1}\rvert^{p}g_{2}(\chi_{\varphi_{1}})}\right\}-C\epsilon^{2}\]
 \end{lem}
 \begin{proof}
 	We calculate 
 	\begin{align*}
	 	&\frac{dE_{p,\delta}^{\epsilon}}{dt}=\int_{X}\frac{p\left|\frac{d\varphi_{\epsilon}}{dt}\right|^{2p-2}\frac{d\varphi_{\epsilon}}{dt}\frac{d^{2}\varphi_{\epsilon}}{dt^{2}}}{\sqrt{\left|\frac{d\varphi_{\epsilon}}{dt}\right|^{2p}+\delta^{2}}}g_{2}(\chi_{\varphi_{\epsilon}})+\int_{X}\sqrt{\left|\frac{d\varphi_{\epsilon}}{dt}\right|^{2p}+\delta^{2}}\, g_{2}'(\chi_{\varphi_{\epsilon}})\wedge\sqrt{-1}\partial\bar{\partial}\left(\frac{d\varphi_{\epsilon}}{dt}\right)\\
	 	&=\int_{X}\frac{p\left|\frac{d\varphi_{\epsilon}}{dt}\right|^{2p-2}\frac{d\varphi_{\epsilon}}{dt}\frac{d^{2}\varphi_{\epsilon}}{dt^{2}}}{\sqrt{\left|\frac{d\varphi_{\epsilon}}{dt}\right|^{2p}+\delta^{2}}}g_{2}(\chi_{\varphi_{\epsilon}})-\int_{X}\frac{p\left|\frac{d\varphi_{\epsilon}}{dt}\right|^{2p-2}\frac{d\varphi_{\epsilon}}{dt}}{\sqrt{\left|\frac{d\varphi_{\epsilon}}{dt}\right|^{2p}+\delta^{2}}}\sqrt{-1}\partial(\frac{d\varphi_{\epsilon}}{dt})\wedge \bar{\partial}(\frac{d\varphi_{\epsilon}}{dt})\wedge g_{2}'(\chi_{\varphi_{\epsilon}})\\
	 	&=\int_{X}\frac{p\left|\frac{d\varphi_{\epsilon}}{dt}\right|^{2p-2}\frac{d\varphi_{\epsilon}}{dt}}{\sqrt{\left|\frac{d\varphi_{\epsilon}}{dt}\right|^{2p}+\delta^{2}}}\left(\frac{d^{2}\varphi_{\epsilon}}{dt^{2}}g_{2}(\chi_{\varphi_{\epsilon}})-\sqrt{-1}\partial(\frac{d\varphi_{\epsilon}}{dt})\wedge \bar{\partial}(\frac{d\varphi_{\epsilon}}{dt})\wedge g_{2}'(\chi_{\varphi_{\epsilon}}) \right).
 	\end{align*}
 Using Lemma \ref{geodesic lemma} and Proposition \ref{ddtestimate}, we obtain the first assertion.
 Since $\pi_{X}^{*}\chi+\sqrt{-1}D\bar D\varphi_{\epsilon}$ is K\"ahler, the $\sqrt{-1}dz_0\wedge d\bar z_0$ part gives the estimate $\frac{d^{2}\varphi_{\epsilon}}{dt^{2}}\ge 0$. This implies that 
 \[\frac{d\varphi_{\epsilon}}{dt}(0)\le \varphi_{\epsilon}(1)-\varphi_{\epsilon}(0)=\varphi_{1}-\varphi_{0}.\]
 Then on the set $\{\varphi_{0}>\varphi_{1}\},$ 
 \[\lvert\frac{d\varphi_{\epsilon}}{dt}(0)\rvert\ge \lvert\varphi_{0}-\varphi_{1}\rvert\]
 and so
 \[E_{p,\delta}^{\epsilon}(0)\ge \int_{\{\varphi_{0}>\varphi_{1}\}}\lvert\varphi_{0}-\varphi_{1}\rvert^{p}g_{2}(\chi_{\varphi_{0}}).\]
 Similarly, we have 
 \[E_{p,\delta}^{\epsilon}(1)\ge \int_{\{\varphi_{1}>\varphi_{0}\}}\lvert\varphi_{0}-\varphi_{1}\rvert^{p}g_{2}(\chi_{\varphi_{1}}).\]
 Combining the above with $(i)$, we obtain $(ii)$.
 \end{proof}
 \begin{lem}
 	Let $\psi(s)$, $s\in[0,s_{0}]$, be a smooth path in $\mathcal{H}$, and let $\hat{\psi}$ be a fixed point in $\mathcal{H}$ such that $\hat{\psi}\notin \psi([0,s_{0}])$. For any $s\in [0,s_{0}]$, let $\varphi_{\epsilon}(s,t)$, $t\in [0,1]$, be the $\epsilon$-geodesic joining $\psi(s)$ to $\hat{\psi}$. Then there exist constants $\epsilon_0(\psi,\hat{\psi},p,\chi,\omega,X)$ and $C(\psi,\hat{\psi},p,\chi,\omega,X)$ such that
	 	\[\operatorname{length}_{p,\delta}(\varphi_{\epsilon}(0,\cdot))\le \operatorname{length}_{p,\delta}(\psi)+\operatorname{length}_{p,\delta}(\varphi_{\epsilon}(s_{0},\cdot))+C\epsilon^{2}\delta^{-3}\]
        for all $\epsilon<\epsilon_0$.
 \end{lem}
 \begin{proof}
 For a fixed $\epsilon$, using ellipticity and the implicit function theorem, we see that $\varphi_{\epsilon}(s,t)$ is differentiable in $s$. 
 Now we write
 \[l_{\delta}(s)=\operatorname{length}_{p,\delta}(\psi|_{[0,s]})\qquad\text{and}\qquad \hat{l}_{\delta}(s)=\operatorname{length}_{p,\delta}(\varphi_{\epsilon}(s,\cdot)).\]
 The lemma follows from the following inequality:
 \[\frac{dl_{\delta}(s)}{ds}+\frac{d\hat{l}_{\delta}(s)}{ds}\ge -C\epsilon^{2}\delta^{-3}.\]
 We see that 
 \begin{equation}
 	\label{l expression}
 \frac{dl_{\delta}(s)}{ds}=E_{p,\delta}^{\frac{1}{p}}(\psi,s)=\left(\int_{X}\sqrt{\left|\frac{d\psi}{ds}\right|^{2p}+\delta^{2}}\,g_{2}(\chi_{\psi})\right)^{\frac{1}{p}}
  \end{equation}
 and 

 \begin{equation}
 	 \label{hat l equality}
 \frac{d\hat{l}_{\delta}(s)}{ds}=\frac{1}{p}\int_{0}^{1}(E_{p,\delta}^{\epsilon}(s,t))^{\frac{1}{p}-1}\partial_{s}E_{p,\delta}^{\epsilon}(s,t)dt.
\end{equation}
 We now compute 
	 	 \begin{equation}
         \begin{split}
	 	&\frac{\partial E_{p,\delta}^{\epsilon}(s,t)}{\partial s}\\
	 	&=p\int_{X}\frac{\left|\frac{d\varphi_{\epsilon}}{dt}\right|^{2p-2}\frac{d\varphi_{\epsilon}}{dt}\frac{d^{2}\varphi_{\epsilon}}{dsdt}}{\sqrt{\left|\frac{d\varphi_{\epsilon}}{dt}\right|^{2p}+\delta^{2}}}g_{2}(\chi_{\varphi_\epsilon})+\int_{X}\sqrt{\left|\frac{d\varphi_{\epsilon}}{dt}\right|^{2p}+\delta^{2}}\,g_{2}'(\chi_{\varphi_\epsilon})\wedge \sqrt{-1}\partial\bar{\partial}\frac{d\varphi_{\epsilon}}{ds}\\
	 	&=p\frac{\partial}{\partial t}\left(\int_{X}\frac{\left|\frac{d\varphi_{\epsilon}}{dt}\right|^{2p-2}\frac{d\varphi_{\epsilon}}{dt}\frac{d\varphi_{\epsilon}}{ds}}{\sqrt{\left|\frac{d\varphi_{\epsilon}}{dt}\right|^{2p}+\delta^{2}}}g_{2}(\chi_{\varphi_\epsilon})\right)-p\int_{X}\frac{d\varphi_{\epsilon}}{ds}\frac{\partial}{\partial t}\left(\frac{\left|\frac{d\varphi_{\epsilon}}{dt}\right|^{2p-2}\frac{d\varphi_{\epsilon}}{dt}}{\sqrt{\left|\frac{d\varphi_{\epsilon}}{dt}\right|^{2p}+\delta^{2}}}g_{2}(\chi_{\varphi_\epsilon})\right)\\
	 	&-p\int_{X}\frac{\left|\frac{d\varphi_{\epsilon}}{dt}\right|^{2p-2}\frac{d\varphi_{\epsilon}}{dt}}{\sqrt{\left|\frac{d\varphi_{\epsilon}}{dt}\right|^{2p}+\delta^{2}}}g_{2}'(\chi_{\varphi_\epsilon})\wedge \sqrt{-1}\partial(\frac{d\varphi_{\epsilon}}{dt})\wedge\bar{\partial}(\frac{d\varphi_{\epsilon}}{ds}) \label{derivative expansion}
        \end{split}
	 \end{equation}
We calculate the second term in \eqref{derivative expansion} directly and obtain
\begin{align*}
	&-p\int_{X}\frac{d\varphi_{\epsilon}}{ds}\frac{\partial}{\partial t}\left(\frac{\left|\frac{d\varphi_{\epsilon}}{dt}\right|^{2p-2}\frac{d\varphi_{\epsilon}}{dt}}{\sqrt{\left|\frac{d\varphi_{\epsilon}}{dt}\right|^{2p}+\delta^{2}}}g_{2}(\chi_{\varphi_\epsilon})\right)\\
	&=-p\int_{X}\frac{d\varphi_{\epsilon}}{ds}\frac{(p-1)\left|\frac{d\varphi_{\epsilon}}{dt}\right|^{4p-2}+(2p-1)\delta^{2}\left|\frac{d\varphi_{\epsilon}}{dt}\right|^{2p-2}}{\left(\left|\frac{d\varphi_{\epsilon}}{dt}\right|^{2p}+\delta^{2}\right)^{\frac{3}{2}}}\frac{d^{2}\varphi_{\epsilon}}{dt^{2}}g_{2}(\chi_{\varphi_{\epsilon}})\\
	&-p\int_{X}\frac{d\varphi_{\epsilon}}{ds}\frac{\left|\frac{d\varphi_{\epsilon}}{dt}\right|^{2p-2}\frac{d\varphi_{\epsilon}}{dt}}{\sqrt{\left|\frac{d\varphi_{\epsilon}}{dt}\right|^{2p}+\delta^{2}}}g_{2}'(\chi_{\varphi_{\epsilon}})\wedge\sqrt{-1}\partial\bar{\partial}(\frac{d\varphi_{\epsilon}}{dt}).
\end{align*}
Integrating the third term in \eqref{derivative expansion} by parts, we obtain
\begin{align*}
	&-p\int_{X}\frac{\left|\frac{d\varphi_{\epsilon}}{dt}\right|^{2p-2}\frac{d\varphi_{\epsilon}}{dt}}{\sqrt{\left|\frac{d\varphi_{\epsilon}}{dt}\right|^{2p}+\delta^{2}}}g_{2}'(\chi_{\varphi_\epsilon})\wedge \sqrt{-1}\partial(\frac{d\varphi_{\epsilon}}{dt})\wedge\bar{\partial}(\frac{d\varphi_{\epsilon}}{ds})\\
	&=p\int_{X}\frac{d\varphi_{\epsilon}}{ds}\frac{(p-1)\left|\frac{d\varphi_{\epsilon}}{dt}\right|^{4p-2}+(2p-1)\delta^{2}\left|\frac{d\varphi_{\epsilon}}{dt}\right|^{2p-2}}{\left(\left|\frac{d\varphi_{\epsilon}}{dt}\right|^{2p}+\delta^{2}\right)^{\frac{3}{2}}}g_{2}'(\chi_{\varphi_{\epsilon}})\wedge\sqrt{-1}\partial(\frac{d\varphi_{\epsilon}}{dt})\wedge\bar{\partial}(\frac{d\varphi_{\epsilon}}{dt})\\
	&+p\int_{X}\frac{\left|\frac{d\varphi_{\epsilon}}{dt}\right|^{2p-2}\frac{d\varphi_{\epsilon}}{dt}\frac{d\varphi_{\epsilon}}{ds}}{\sqrt{\left|\frac{d\varphi_{\epsilon}}{dt}\right|^{2p}+\delta^{2}}}\sqrt{-1}\partial\bar{\partial}(\frac{d\varphi_{\epsilon}}{dt})\wedge g_{2}'(\chi_{\varphi_{\epsilon}}).
\end{align*}
Substituting these formulas into \eqref{derivative expansion} and using Lemma~\ref{geodesic lemma}, we obtain
\begin{align*}
			\frac{\partial E_{p,\delta}^{\epsilon}(s,t)}{\partial s}
			&=p\frac{\partial}{\partial t}\left(\int_{X}\frac{\left|\frac{d\varphi_{\epsilon}}{dt}\right|^{2p-2}\frac{d\varphi_{\epsilon}}{dt}\frac{d\varphi_{\epsilon}}{ds}}{\sqrt{\left|\frac{d\varphi_{\epsilon}}{dt}\right|^{2p}+\delta^{2}}}g_{2}(\chi_{\varphi_\epsilon})\right)\\
			&+p\int_{X}4\epsilon^{2}e^{2t}\frac{d\varphi_{\epsilon}}{ds}\frac{(p-1)\left|\frac{d\varphi_{\epsilon}}{dt}\right|^{4p-2}+(2p-1)\delta^{2}\left|\frac{d\varphi_{\epsilon}}{dt}\right|^{2p-2}}{\left(\left|\frac{d\varphi_{\epsilon}}{dt}\right|^{2p}+\delta^{2}\right)^{\frac{3}{2}}}\tilde g_2(\chi_\varphi).
\end{align*}
Applying $\partial_{s}$ to the $\epsilon$-geodesic equation and using the maximum principle, we obtain $\lvert \frac{d\varphi_{\epsilon}}{ds}\rvert\le C$. Using Proposition \ref{ddtestimate}, we obtain
\[\frac{\partial E_{p,\delta}^{\epsilon}(s,t)}{\partial s}\ge p\frac{\partial}{\partial t}\left(\int_{X}\frac{\left|\frac{d\varphi_{\epsilon}}{dt}\right|^{2p-2}\frac{d\varphi_{\epsilon}}{dt}\frac{d\varphi_{\epsilon}}{ds}}{\sqrt{\left|\frac{d\varphi_{\epsilon}}{dt}\right|^{2p}+\delta^{2}}}g_{2}(\chi_{\varphi_\epsilon})\right)- C\epsilon^{2}\delta^{-3}.\]
Observe that in the above the constant $C$ depends on the cohomological constant $\int_{X}|\tilde{g_{2}}(\chi_{\varphi})|$.
Substituting this into \eqref{hat l equality}, we obtain
\begin{align*}
	&\frac{d\hat{l}_{\delta}(s)}{ds}=\frac{1}{p}\int_{0}^{1}(E_{p,\delta}^{\epsilon}(s,t))^{\frac{1}{p}-1}\partial_{s}E_{p,\delta}^{\epsilon}(s,t)dt\\
	&\ge \int_{0}^{1}(E_{p,\delta}^{\epsilon}(s,t))^{\frac{1}{p}-1}\frac{\partial}{\partial t}\left(\int_{X}\frac{\left|\frac{d\varphi_{\epsilon}}{dt}\right|^{2p-2}\frac{d\varphi_{\epsilon}}{dt}\frac{d\varphi_{\epsilon}}{ds}}{\sqrt{\left|\frac{d\varphi_{\epsilon}}{dt}\right|^{2p}+\delta^{2}}}g_{2}(\chi_{\varphi_\epsilon})\right)dt-\int_{0}^{1}(E_{p,\delta}^{\epsilon}(s,t))^{\frac{1}{p}-1}C\epsilon^{2}\delta^{-3}dt\\
	&=\left.(E_{p,\delta}^{\epsilon}(s,t))^{\frac{1}{p}-1}\int_{X}\frac{\left|\frac{d\varphi_{\epsilon}}{dt}\right|^{2p-2}\frac{d\varphi_{\epsilon}}{dt}\frac{d\varphi_{\epsilon}}{ds}}{\sqrt{\left|\frac{d\varphi_{\epsilon}}{dt}\right|^{2p}+\delta^{2}}}g_{2}(\chi_{\varphi_\epsilon})\right|_{t=0}^{t=1}\\
	&-\int_{0}^{1}\left(\frac{\partial}{\partial t}(E_{p,\delta}^{\epsilon}(s,t))^{\frac{1}{p}-1}\right)\left(\int_{X}\frac{\left|\frac{d\varphi_{\epsilon}}{dt}\right|^{2p-2}\frac{d\varphi_{\epsilon}}{dt}\frac{d\varphi_{\epsilon}}{ds}}{\sqrt{\left|\frac{d\varphi_{\epsilon}}{dt}\right|^{2p}+\delta^{2}}}g_{2}(\chi_{\varphi_\epsilon})\right)dt-\int_{0}^{1}(E_{p,\delta}^{\epsilon}(s,t))^{\frac{1}{p}-1}C\epsilon^{2}\delta^{-3}dt.
\end{align*}
We know that $\varphi_{\epsilon}(s,0)=\psi(s)$ and $\varphi_{\epsilon}(s,1)=\hat{\psi}$. Hence, $\frac{d\varphi_{\epsilon}}{ds}(s,0)=\frac{d\psi}{ds}(s)$ and $\frac{d\varphi_{\epsilon}}{ds}(s,1)=0$. The second term and the third term can be controlled in the following way. Note that the function
\[
\max\Big\{\int_{\{\psi(s)>\hat\psi\}}|\psi(s)-\hat\psi|^p g_2(\chi_{\psi(s)}), \int_{\{\hat\psi>\psi(s)\}}|\psi(s)-\hat\psi|^p g_2(\chi_{\hat\psi})\Big\}
\]
is continuous in $s$, and it is positive because $\hat{\psi}\notin \psi([0,s_{0}])$. So it has a uniform positive lower bound. Thus, by Lemma \ref{estimate lemma}(ii), for all sufficiently small $\epsilon$, $E_{p,\delta}^{\epsilon}(s,t)$ also has a uniform positive lower bound.
Combining this with Lemma \ref{estimate lemma}(i) and the bounds on $\lvert\frac{d\varphi_\epsilon}{dt}\rvert$ and $\lvert\frac{d\varphi_\epsilon}{ds}\rvert$, we obtain
\[\frac{d\hat{l}_{\delta}(s)}{ds}\ge -\left.(E_{p,\delta}^{\epsilon}(s,0))^{\frac{1}{p}-1}\int_{X}\frac{\left|\frac{d\varphi_{\epsilon}}{dt}\right|^{2p-2}\frac{d\varphi_{\epsilon}}{dt}\frac{d\varphi_{\epsilon}}{ds}}{\sqrt{\left|\frac{d\varphi_{\epsilon}}{dt}\right|^{2p}+\delta^{2}}}g_{2}(\chi_{\varphi_\epsilon})\right|_{t=0}-C\epsilon^{2}\delta^{-3}.\]
Using H\"older's inequality, when $t=0$, we obtain
\begin{align*}
	&\int_{X}\frac{\left|\frac{d\varphi_{\epsilon}}{dt}\right|^{2p-2}\frac{d\varphi_{\epsilon}}{dt}\frac{d\varphi_{\epsilon}}{ds}}{\sqrt{\left|\frac{d\varphi_{\epsilon}}{dt}\right|^{2p}+\delta^{2}}}g_{2}(\chi_{\varphi_\epsilon})\\
	&\le \left(\int_{X}\lvert\frac{d\varphi_{\epsilon}}{ds}\rvert^{p}g_{2}(\chi_{\varphi_{\epsilon}})\right)^{\frac{1}{p}}\left(\int_{X}\left|\frac{\left|\frac{d\varphi_{\epsilon}}{dt}\right|^{2p-2}\frac{d\varphi_{\epsilon}}{dt}}{\sqrt{\left|\frac{d\varphi_{\epsilon}}{dt}\right|^{2p}+\delta^{2}}}\right|^{q}g_{2}(\chi_{\varphi_{\epsilon}})\right)^{\frac{1}{q}}\\
	&\le \left(\int_{X}\lvert\frac{d\psi}{ds}\rvert^{p}g_{2}(\chi_{\varphi_{\epsilon}})\right)^{\frac{1}{p}}\times (E_{p,\delta}^{\epsilon}(s,0))^{1-\frac{1}{p}}
\end{align*}
where $q=p/(p-1)$ when $p>1$. The last inequality follows pointwise from
\[
\left|\frac{|v|^{2p-2}v}{\sqrt{|v|^{2p}+\delta^2}}\right|^q
\le \sqrt{|v|^{2p}+\delta^2}.
\]
When $p=1$, the same conclusion follows from
\[
\frac{|v|}{\sqrt{|v|^2+\delta^2}}\le 1.
\]
Consequently,
\[\frac{d\hat{l}_{\delta}(s)}{ds}\ge -\left(\int_{X}\lvert\frac{d\psi}{ds}\rvert^{p}g_{2}(\chi_{\varphi_{\epsilon}})\right)^{\frac{1}{p}}-C\epsilon^{2}\delta^{-3}.\]
Using \eqref{l expression}, we obtain
\[\frac{dl_{\delta}(s)}{ds}+\frac{d\hat{l}_{\delta}(s)}{ds}\ge -C\epsilon^{2}\delta^{-3}.\]

 \end{proof}

 \begin{thm}
 \label{dp-distance-limit}
	 	For any $\varphi_{0},\varphi_{1}\in \mathcal{H}$, let $\varphi_{\epsilon}$ be the $\epsilon$-geodesic connecting $\varphi_{0}$ and $\varphi_{1}$. Then:
	 	\begin{enumerate}
	 	\item $d_{p}(\varphi_{0},\varphi_{1})=\lim_{\epsilon\rightarrow 0}\operatorname{length}_{p}(\varphi_{\epsilon})$;
	 	\item $d_{p}^{p}(\varphi_{0},\varphi_{1})=\lim_{\epsilon\rightarrow 0}E_{p}^{\epsilon}(t)$ for every $t\in [0,1]$;
	 	\item The following inequality holds:
\[
d_p^p(\varphi_0,\varphi_1) \ge \max\Big\{
\int_{\{\varphi_0\ge\varphi_1\}} |\varphi_0-\varphi_1|^p g_2(\chi_{\varphi_0}),\;
\int_{\{\varphi_1\ge\varphi_0\}} |\varphi_1-\varphi_0|^p g_2(\chi_{\varphi_1})
\Big\}.
\]
	 	\end{enumerate}
\end{thm}
\begin{proof}
We first prove (i). If $\varphi_0=\varphi_1$, then $d_{p}(\varphi_{0},\varphi_{1})=0$ by taking $\varphi_t=\varphi_0=\varphi_1$, and $\lim_{\epsilon\rightarrow 0}\operatorname{length}_{p}(\varphi_{\epsilon})=0$ by Proposition \ref{ddtestimate}. So we assume that $\varphi_0\not=\varphi_1$.
By definition,
\[d_{p}(\varphi_{0},\varphi_{1})\le \liminf_{\epsilon\rightarrow 0}\operatorname{length}_{p}(\varphi_{\epsilon}).\]
It therefore remains to prove
\[\limsup_{\epsilon\rightarrow 0}\operatorname{length}_{p}(\varphi_{\epsilon})\le d_{p}(\varphi_{0},\varphi_{1}).\]
Let $\psi(s)$, $s\in [0,1]$, be any smooth path connecting $\varphi_{0}$ and $\varphi_{1}$. After removing any portion of the path following its first arrival at $\varphi_1$, we may assume that $\varphi_{1}\notin \psi([0,1))$. For each $s\in [0,1)$, let $\varphi_{\epsilon}(s,t)$ be the $\epsilon$-geodesic connecting $\psi(s)$ to $\varphi_{1}$. For any $s_{0}\in (0,1)$, the preceding lemma gives
		\[\operatorname{length}_{p,\delta}(\varphi_{\epsilon}(0,\cdot))\le \operatorname{length}_{p,\delta}(\psi|_{[0,s_{0}]})+\operatorname{length}_{p,\delta}(\varphi_{\epsilon}(s_{0},\cdot))+C\epsilon^{2}\delta^{-3}.\]
		Using \eqref{approximate inequality} and $\varphi_{\epsilon}(0,\cdot)=\varphi_{\epsilon}$, we obtain
		\[\operatorname{length}_{p}(\varphi_{\epsilon})\le \operatorname{length}_{p}(\psi|_{[0,s_{0}]})+\operatorname{length}_{p}(\varphi_{\epsilon}(s_{0},\cdot))+C\delta^{\frac{1}{p}}+C\epsilon^{2}\delta^{-3}.\]
		Since Proposition \ref{ddtestimate} gives $\lvert\frac{d\varphi_{\epsilon}}{dt}\rvert\le C\|\psi(s_0)-\varphi_{1}\|_{L^{\infty}}+C\epsilon^{2}$, we obtain
		\[\operatorname{length}_{p}(\varphi_{\epsilon}(s_{0},\cdot))\le C\|\psi(s_{0})-\psi(1)\|_{L^{\infty}}+C\epsilon^{2}.\]
		Hence,
		\[\limsup_{\epsilon\rightarrow 0}\operatorname{length}_{p}(\varphi_{\epsilon})\le \operatorname{length}_{p}(\psi|_{[0,s_{0}]})+C\|\psi(s_{0})-\psi(1)\|_{L^{\infty}}+C\delta^{\frac{1}{p}}.\]
		Letting $\delta\rightarrow 0$ and then $s_{0}\rightarrow 1$, we obtain
		\[\limsup_{\epsilon\rightarrow 0}\operatorname{length}_{p}(\varphi_{\epsilon})\le \operatorname{length}_{p}(\psi).\]
		Since $\psi$ is an arbitrary path connecting $\varphi_{0}$ and $\varphi_{1}$,
		\[\limsup_{\epsilon\rightarrow 0}\operatorname{length}_{p}(\varphi_{\epsilon})\le d_{p}(\varphi_{0},\varphi_{1}).\]
		This proves (i). Assertion (ii) follows from (i), Lemma~\ref{estimate lemma}, and \eqref{approximate inequality}, after letting $\delta\to0$. Assertion (iii) follows from Lemma~\ref{estimate lemma} by letting $\epsilon\to0$ and $\delta\to0$.
\end{proof}

As a corollary, we prove that $(\mathcal{H},d_p)$ is a metric space.
\begin{cor}
$(\mathcal{H},d_p)$ is a metric space.
\end{cor}
\begin{proof}
The triangle inequality, symmetry and positivity are standard. Since $\mathcal{H}$ is path connected, the reversed path proves symmetry and the glued path proves the triangle inequality. The only non-trivial thing is that $d_p(\varphi_0,\varphi_1)=0$ implies that $\varphi_0=\varphi_1$. This is true by Theorem \ref{dp-distance-limit} (3).
\end{proof}
Then we can define the $\mathcal{E}^p$ space as the completion of $(\mathcal{H},d_p)$. For $f_{2}(x)=x^{n}$, see \cite{DarvasMabuchiCompletion} for more details.

\section{Examples}

The classical example \cite{ChenMabuchiEnergy} is $f_1(x)=x^n-cx^{n-1}$, $c>0$ and $f_2(x)=x^n$, which says that the $J$-functional is convex along classical geodesics. For the hypercritical LYZ equation, the Collins-Yau theory \cite{CollinsYauGeodesics} corresponds to
\[
        f_1(x)=\Im(e^{-\sqrt{-1}\hat\theta}(1+\sqrt{-1}x)^n), \quad f_2(x)=\Re(e^{-\sqrt{-1}\hat\theta}(1+\sqrt{-1}x)^n).
\]
If $\hat\theta$ is in the supercritical phase $(\frac{n-2}{2}\pi, \frac{n}{2}\pi)$, then for any $\frac{n-2}{2}\pi<\theta_0<\hat\theta<\frac{n}{2}\pi$, the polynomials
\[
        f_1(x)=\Im(e^{-\sqrt{-1}\hat\theta}(1+\sqrt{-1}x)^n), \quad f_2(x)=\Im(e^{-\sqrt{-1}\theta_0}(1+\sqrt{-1}x)^n)
\]
also satisfy all of our conditions. This reduces to the Collins-Yau case if $\hat\theta$ is in the hypercritical phase $(\frac{n-1}{2}\pi, \frac{n}{2}\pi)$ and $\theta_0=\hat\theta-\frac{\pi}{2}$.

Another case is that $f_1(x)=x^n-c_kx^k$, $0\leq k\leq n-1$, $c_k>0$ and $f_2(x)=x^n$. In this case, the geodesic is the classical one and Fang-Lai-Ma \cite{FangLaiMaFlows} have studied the functional and its gradient flow.

\bibliographystyle{plain}
\bibliography{reference}

\begin{center}
\begin{minipage}{0.88\textwidth}
\small
\textbf{Gao Chen}\\[0.35em]
University of Science and Technology of China\\
No.\ 96 Jinzhai Road, Baohe District, Hefei, Anhui 230000, P.\,R.\ China\\[0.35em]
\href{mailto:chengao1@ustc.edu.cn}{\texttt{chengao1@ustc.edu.cn}}\\
\textbf{Kartick Ghosh}\\[0.35em]
University of Science and Technology of China\\
No.\ 96 Jinzhai Road, Baohe District, Hefei, Anhui 230000, P.\,R.\ China\\[0.35em]
\href{mailto:kghosh@ustc.edu.cn}{\texttt{kghosh@ustc.edu.cn}}\\
\textbf{Sijie Nie}\\[0.35em]
University of Science and Technology of China\\
No.\ 96 Jinzhai Road, Baohe District, Hefei, Anhui 230000, P.\,R.\ China\\[0.35em]
\href{mailto:niesijie@mail.ustc.edu.cn}{\texttt{niesijie@mail.ustc.edu.cn}}
\end{minipage}
\end{center}

\end{document}